\documentclass[dvipsnames]{article}
\usepackage{amsmath,amssymb,amsthm,bm}
\usepackage[a4paper]{geometry}
\usepackage{algorithm,algorithmic}
\usepackage{caption}
\usepackage{graphicx}
\usepackage[T1]{fontenc}
\usepackage{multirow}
\usepackage{nicematrix}
\usepackage{authblk}
\usepackage{cite}

\usepackage{booktabs}

\usepackage[colorlinks,
linkcolor=MidnightBlue,
citecolor=ForestGreen]{hyperref}
\usepackage{cleveref}
\usepackage{aliascnt}

\graphicspath{{figs/}}

\newenvironment{keywords}{\medskip\textbf{Keywords:}}{}

\newenvironment{MSCcodes}{\medskip\textbf{AMS subject classifications (2020).}}{}

\theoremstyle{plain}
\newtheorem{theorem}{Theorem}

\newaliascnt{lemma}{theorem}
\newtheorem{lemma}[lemma]{Lemma}
\aliascntresetthe{lemma}

\newaliascnt{proposition}{theorem}
\newtheorem{proposition}[proposition]{Proposition}
\aliascntresetthe{proposition}

\newaliascnt{corollary}{theorem}
\newtheorem{corollary}[corollary]{Corollary}
\aliascntresetthe{corollary}

\theoremstyle{definition}

\theoremstyle{remark}
\newtheorem{remark}{Remark}

\definecolor{bkgndcolor}{rgb}{1,1,1}
\definecolor{bkgndcolor}{rgb}{0.78,0.93,0.8}

\usepackage{tcolorbox}
\tcbuselibrary{skins}
\tcbuselibrary{breakable}
\tcbset{shield externalize}

\DeclareFontFamily{U}{mathx}{}
\DeclareFontShape{U}{mathx}{m}{n}{<-> mathx10}{}
\DeclareSymbolFont{mathx}{U}{mathx}{m}{n}
\DeclareMathAccent{\widecheck}{0}{mathx}{"71}
\makeatother

\DeclareMathOperator{\cond}{\kappa}
\DeclareMathOperator{\diag}{diag}
\DeclareMathOperator{\Null}{Null}

\DeclareMathOperator{\Span}{span}
\DeclareMathOperator{\trace}{tr}

\DeclareMathOperator{\nnz}{nnz}
\DeclareMathOperator{\sv}{\sigma}

\DeclareMathOperator{\dist}{dist}

\renewcommand{\Re}{\mathop{\mathrm{Re}}}

\newcommand{\fro}{\mathsf F}

\newcommand*{\innerp}[2]{\left\langle#1, #2\right\rangle}
\newcommand*{\trans}{^{\top}}

\newcommand*{\pinv}{^{\dagger}}

\newcommand*{\abs}[1]{\lvert#1\rvert}
\newcommand*{\norm}[1]{\lVert#1\rVert}

\def\adots{\mathinner{\mkern2mu\raise1pt\hbox{.}\mkern2mu
    \raise4pt\hbox{.}\mkern2mu\raise7pt\hbox{.}\mkern1mu}}

\newcommand{\subroutine}[1]{\texttt{#1}}
\newcommand{\algorithmicfunction}{\textbf{Function:}}
\newcommand{\FUNCTION}[1]{\item[\algorithmicfunction]#1}

\newcommand*{\Mat}[1]{\bm{#1}}
\newcommand*{\tMat}[1]{\widetilde{\Mat{#1}}}

\newcommand*{\bigO}{\mathcal{O}}

\newcommand*{\Subspace}[1]{\mathcal{#1}}

\newcommand*{\bbR}{\mathbb{R}}
\newcommand*{\bbC}{\mathbb{C}}

\newcommand*{\zeromat}{\Mat{0}}
\newcommand*{\I}{\Mat{I}}
\newcommand*{\A}{\Mat{A}}
\newcommand*{\tA}{\tMat{A}}

\newcommand*{\V}{\Mat{V}}
\newcommand*{\lam}{\lambda}
\newcommand*{\Lam}{\Mat{\Lambda}}

\newcommand*{\X}{\Mat{X}}

\newcommand*{\AX}{\Mat{F}}

\newcommand*{\tX}{\widetilde{\X}}
\newcommand*{\tAX}{\widetilde{\AX}}

\newcommand*{\n}{m}
\newcommand*{\subspacedim}{n}
\newcommand*{\nev}{n_{\mathrm{ev}}}
\newcommand*{\nex}{n_{\mathrm{ex}}}

\newcommand*{\mynev}{n_{\mathrm{ev, now}}}

\renewcommand*{\P}{\Mat{P}}
\newcommand*{\subspaceU}{\Subspace{U}}
\newcommand*{\U}{\Mat{U}}

\newcommand*{\skmat}{\Mat{S}}
\newcommand*{\skdim}{s}
\newcommand*{\sksubspace}{\Subspace{U}}
\newcommand*{\skvec}{\Mat{u}}
\newcommand*{\skeps}{\varepsilon}
\newcommand*{\skinnerp}[2]{\innerp{#1}{#2}_{\skmat}}
\newcommand*{\sknorm}[1]{\norm{#1}_{\skmat}}
\newcommand*{\skperp}{\perp_{\skmat}}

\newcommand*{\skX}{\tX}
\newcommand*{\skAX}{\tAX}

\newcommand*{\skwhX}{\widetilde{\X}_{\skmat}}
\newcommand*{\whX}{\XS}
\newcommand*{\whiteingR}{\Mat{G}}

\newcommand*{\constQR}{c_{\mathrm{QR}}}

\newcommand*{\Z}{\Mat{Z}}

\newcommand*{\Th}{\Mat{\Theta}}
\newcommand*{\T}{\Mat{T}}

\newcommand*{\Ap}{\A_{\mathrm{p}}}
\newcommand*{\Apsk}{\tA_{\mathrm{p}}}

\newcommand*{\R}{\Mat{R}}
\newcommand*{\tR}{\tMat{R}}
\newcommand*{\skR}{\tR}

\newcommand*{\resnorm}{\rho}
\newcommand*{\skresnorm}{\rho_{\skmat}}

\newcommand*{\M}{\Mat{M}}

\newcommand*{\used}{\mathrm{used}}

\newcommand*{\Rused}{\R_{\used}}
\newcommand*{\skRused}{\tR_{\used}}

\newcommand*{\tol}{\tau}

\newcommand*{\QRfunc}{\subroutine{qr\_econ}}
\newcommand*{\normestfunc}{\subroutine{normest}}

\newcommand*{\RRfunc}{\subroutine{RR}}
\newcommand*{\sRRfunc}{\subroutine{sRR\_schur}}

\newcommand*{\SIfunc}{\subroutine{SI}}
\newcommand*{\sSIfunc}{\subroutine{sSI}}

\newcommand*{\TrMINfunc}{\subroutine{TraceMIN}}
\newcommand*{\sTrMINfunc}{\subroutine{sTraceMIN}}
\newcommand*{\TrMINP}{\Mat{\Pi}}
\newcommand*{\TrMINDel}{\Mat{\Delta}}
\newcommand*{\TrMINPsk}{\TrMINP_{\skmat}}

\newcommand*{\iterind}[1]{^{[#1]}}

\newcommand*{\C}{\Mat{C}}
\newcommand*{\XS}{\X_{\skmat}}
\newcommand*{\RS}{\R_{\skmat}}

\newcommand*{\sklower}{\alpha_{\skeps}}
\newcommand*{\skupper}{\beta_{\skeps}}
\newcommand*{\skspnorm}{L}
\newcommand*{\x}{\Mat{x}}

\newcommand*{\E}{\Mat{E}}
\newcommand*{\ubE}{\gamma}
\newcommand*{\skepsneighbor}{\skeps_{*}}
\newcommand*{\ubEneighbor}{\gamma_{*}}

\renewcommand*{\c}{\Mat{c}}
\newcommand*{\z}{\Mat{z}}
\newcommand*{\unitvec}{\Mat{e}}
\renewcommand*{\v}{\Mat{v}}

\newcommand*{\projVX}{\Mat{Q}}

\newcommand*{\TrMINDelS}{\TrMINDel_{\skmat}}
\newcommand*{\Y}{\Mat{Y}}
\newcommand*{\YS}{\Y_{\skmat}}

\newcommand*{\Perturb}{\Mat{Y}}
\newcommand*{\ZS}{\Z_{\skmat}}

\newcommand*{\evgap}{\delta}
\newcommand*{\Perturbnorm}{\eta}
\newcommand*{\PerturbVnorm}{\eta_{1}}
\newcommand*{\N}{\Mat{N}}
\newcommand{\mygap}{d}

\newcommand*{\nrhs}{n_{\mathrm{rhs}}}
\newcommand*{\W}{\Mat{W}}
\newcommand*{\tW}{\tMat{W}}

\title{Real-Arithmetic Sketched Eigensolvers for Symmetric Matrices}

\author[1]{Jingyu Liu}
\author[2]{Yuqi Liu}
\affil[1]{School of Mathematical Sciences, Fudan University, 220 Handan Road, Shanghai 200433, China. jyliu22@m.fudan.edu.cn}
\affil[2]{School of Computer Science, University of Leeds, Woodhouse Lane, Leeds LS2 9JT, United Kingdom. Y.Liu12@leeds.ac.uk}
\date{\today}

\begin{document}

\maketitle

\begin{abstract}
In this work, we develop sketched variants of subspace iteration
and TraceMIN, called sSI and sTraceMIN, for real symmetric eigenvalue
problems by combining basis whitening with the sketched
Rayleigh--Ritz~(sRR) procedure.
The nonsymmetric projected matrix in sRR may have nonreal eigenpairs,
leading to complex arithmetic and additional computational costs.
We overcome this difficulty using a real Schur formulation of sRR
that keeps the iterations in real arithmetic.
We further develop procedures for characteristic subspace selection
and convergence monitoring within this formulation.
On the theoretical side, we discuss suitable subspace embedding assumptions
for sketched eigensolvers that use a fixed sketching matrix throughout
the iterations.
Under a mild embedding assumption, we establish local convergence results for sSI and sTraceMIN.
Numerical experiments demonstrate runtime savings of up to \(39\%\) for sSI and \(21\%\)
for sTraceMIN relative to their standard counterparts on the test problems.
\end{abstract}



\section{Introduction}
\label{sec:introduction}

In this work, we consider the symmetric eigenvalue problem~(EVP)
\begin{equation*}
\A \V = \V \Lam,
\end{equation*}
where \(\A \in \bbR^{\n \times \n}\) is a symmetric matrix,
\(\Lam \in \bbR^{\nev \times \nev}\) is a diagonal matrix
whose diagonal entries are the eigenvalues of interest,
and the columns of \(\V \in \bbR^{\n \times \nev}\) are the corresponding eigenvectors.
We focus primarily on large sparse matrices with \(\nev \ll \n\).
In particular, the algorithms and convergence analysis developed
in this paper target the \(\nev\) smallest eigenpairs of symmetric
positive definite~(SPD) matrices.

As one of the fundamental problems in numerical linear algebra,
the large-scale symmetric eigenvalue problem
has been extensively studied over the past several decades.
Representative algorithms include the subspace iteration~(SI) algorithm~\cite{Demmel_1997},
the Lanczos algorithm~\cite{Demmel_1997},
the Jacobi--Davidson~(JD) algorithm~\cite{Sleijpen_Van_der_Vorst_2000}, 
the locally optimal block preconditioned conjugate gradient~(LOBPCG) algorithm~\cite{Duersch_Shao_Yang_Gu_2018},
and the trace minimization (TraceMIN) algorithm~\cite{Sameh_Wisniewski_1982}.
These algorithms construct approximations to the desired eigenspace in different ways.

In these algorithms, a search subspace is represented by a matrix \(\X \in \bbR^{\n \times \subspacedim}\),
from which approximate eigenpairs are extracted using the Rayleigh--Ritz~(RR) procedure.
Maintaining a well-conditioned basis is important for the numerical stability of these computations.
This is commonly achieved by orthonormalizing \(\X\) before the RR procedure.

For large-scale problems with \(\n \gg 1\), however, 
orthogonalization can be very expensive 
and may account for a substantial fraction of the overall computational cost.
Reducing the orthogonalization cost in large-scale eigensolvers
therefore has received considerable attention.
Common strategies include 
reducing the frequency of orthogonalization in SI~\cite[section~5.1]{Saad_2011} 
and reducing the number of vectors that need to be orthogonalized by, for example, 
locking converged eigenvectors~\cite{Lehoucq_Sorensen_1996, Saad_2011}.

Randomized sketching provides a way to reduce the cost of orthogonalization
while controlling the conditioning of the basis.
Examples include the randomized Gram--Schmidt~(RGS) process~\cite{Balabanov_Grigori_2022, Balabanov_Grigori_2025}
and, more generally, basis whitening~\cite{Rokhlin_Tygert_2008, Nakatsukasa_Tropp_2024}.
Whitening constructs a sketch-orthonormal basis
\(\X \in \bbR^{\n \times \subspacedim}\),
i.e., \(\skX = \skmat \X\) is orthonormal.
Here, \(\skmat \in \bbR^{\skdim \times \n}\) is a sketching
matrix with \(\subspacedim \le \skdim \ll \n\).
For suitably chosen \(\skmat\), the sketch-orthonormal basis \(\X\)
is well-conditioned with high probability, while forming
\(\skmat\X\) remains inexpensive.
Hence, whitening can be much cheaper 
than standard QR factorization.
When the application of the sketching matrix
and computations in the sketched space 
contribute only lower-order costs,
its leading cost is approximately 
half that of classical or modified 
Gram--Schmidt~(CGS or MGS), 
and one quarter that of Householder QR 
with the orthonormal factor formed explicitly.

Nevertheless, the savings from whitening
may be offset by the subsequent RR procedure.
For a nonorthonormal basis \(\X\), 
the conventional RR procedure requires 
either explicitly orthonormalizing \(\X\) 
or solving the generalized eigenvalue problem~(GEVP)
\begin{equation*}
(\X\trans \A \X) \Z = (\X\trans \X) \Z \Th.
\end{equation*}
The former requires computing a standard QR factorization of \(\X\),
whereas the latter requires explicitly forming the Gram matrix \(\X\trans\X\).
Either operation involves inner products in dimension \(\n\)
and can erase the computational savings achieved by whitening.
A related trade-off has been observed in the recently proposed randomized block Davidson eigensolver~\cite{Gubler_Park_Bussy_Laute_Marzari_Timrov_Grigori_2026}.

Fortunately, the sketched RR~(sRR) procedure~\cite{Nakatsukasa_Tropp_2024} offers a way to use the basis \(\X\)
without orthogonalizing it or forming \(\X\trans \X\).
It replaces the Euclidean projection by a sketched one.
For a sketch-orthonormal basis \(\X\), its projected matrix is
\begin{equation*}
\Apsk = (\skmat \X)\trans (\skmat \A \X).
\end{equation*}
This construction replaces inner products in the original space by
cheaper ones in the sketched space.
Moreover, since \(\skmat\X\) is already computed during basis whitening, it can be reused without additional computation.
This makes it natural to combine basis whitening and sRR in eigensolvers.
In a recent work~\cite{Grigori_Park_Simunec_2026}, the authors applied this idea
to the JD algorithm for nonsymmetric EVPs.

However, such a direct combination may not work efficiently 
for the symmetric EVPs considered here.
Although \(\A\) is SPD, \(\Apsk\) is generally nonsymmetric.
As a result, the approximate eigenpairs produced by sRR may be complex.
For iterative eigensolvers, this may force 
all subsequent computations into complex arithmetic,
again offsetting the computational savings achieved by whitening.
This difficulty was also noted in~\cite[section~6.8]{Nakatsukasa_Tropp_2024}.

To keep the algorithms in real arithmetic, we use the real Schur decomposition,
\begin{equation*}
\Apsk = \Z \T \Z\trans,
\end{equation*}
instead of a spectral decomposition in sRR.
Here, \(\Z\) is orthogonal and \(\T\) is upper quasi-triangular.
The transformed basis \(\X \Z\) remains real and sketch-orthonormal,
allowing search subspaces to be represented and updated without forming complex eigenvectors.

The main algorithmic contribution of this paper is a framework
that combines whitening and the real Schur formulation of sRR
with the subspace updates of iterative eigensolvers,
while keeping the large-scale iterations in real arithmetic.
We apply this framework to SI and TraceMIN, obtaining the
sketched SI~(sSI) and sketched TraceMIN~(sTraceMIN) algorithms.
Several implementation issues that arise from the use of the real Schur form,
including
the extraction of the target characteristic subspace and the estimation of the corresponding residuals,
are also discussed in detail.
Numerical experiments compare the proposed methods with their
conventional counterparts in terms of convergence behavior,
accuracy, and computational cost.

We also provide theoretical analysis of the proposed methods.
Typically, the analysis of sketched eigensolvers assumes that
\(\skmat\) is a subspace embedding for \(\Span\{\X,\A\X\}\), or at least \(\Span\{\X\}\), throughout all iterations.
However, for the proposed eigensolvers, the sketching matrix \(\skmat\) is fixed throughout the iterations while \(\X\) can typically trial whole \(\bbR^{\n}\). 
Therefore such embedding assumptions are no longer a suitable choice.
Nevertheless, under a mild local embedding assumption, we show that sSI and sTraceMIN
generate the same search subspaces as their standard counterparts,
when the initial search subspaces are sufficiently 
close to the characteristic subspace spanned by the \(\nex\) target eigenvectors.
Additionally, we discuss the subspaces selected by real Schur sRR and RR,
deriving an upper bound on the gap between these two subspaces locally.

The remainder of this paper is organized as follows.
Section~\ref{sec:preliminaries} introduces the notation used throughout the paper and
reviews randomized sketching and basis whitening.
Section~\ref{sec:sketched_eigensolvers} presents the proposed algorithms,
and Section~\ref{sec:implementation} discusses their implementation.
Section~\ref{sec:convergence} presents the convergence analysis of the proposed algorithms.
Section~\ref{sec:experiments} presents numerical experiments.
Finally, Section~\ref{sec:conclusions} summarizes the results and outlines directions for future work.

\section{Preliminaries} \label{sec:preliminaries}

\subsection{Notation} \label{subsec:notation}

We use \(\norm{\cdot}_{2}\) to denote the Euclidean norm of a vector
or the spectral norm of a matrix, and \(\norm{\cdot}_{\fro}\)
to denote the Frobenius norm.
The Euclidean inner product is denoted by
\(\innerp{\Mat{x}}{\Mat{y}} = \Mat{y}\trans \Mat{x}\).
We use \(\sv_{\min}(\cdot)\) and \(\sv_{\max}(\cdot)\) to denote the smallest and largest singular values of a matrix, respectively.
The notation \(\cond_{2}(\cdot)\) denotes the condition number of a matrix with respect to the spectral norm.
The dagger symbol \(\pinv\) denotes the pseudoinverse.
A matrix \(\X \in \bbR^{\n \times \subspacedim}\) is called
orthonormal if \(\X\trans \X = \I_{\subspacedim}\).
Given a sketching matrix \(\skmat\), the sketched ``inner product'' and sketched ``norm'' 
are defined as \(\skinnerp{\Mat{x}}{\Mat{y}} = \innerp{\skmat \Mat{x}}{\skmat \Mat{y}}\) 
and \(\sknorm{\Mat{x}} = \sqrt{\skinnerp{\Mat{x}}{\Mat{x}}}\), respectively.
The notation \(\skperp\) indicates orthogonality with respect to the sketched inner product.
A matrix \(\X\) is called sketch-orthonormal if \(\skmat \X\) is orthonormal.
The distance between two subspaces
\(\subspaceU_{1}, \subspaceU_{2} \subseteq \bbR^{\n}\)
is defined as
\(\dist(\subspaceU_{1}, \subspaceU_{2}) = \|\P_{1} - \P_{2}\|_{2}\),
where \(\P_{1}\) and \(\P_{2}\) are the orthogonal projectors onto
\(\subspaceU_{1}\) and \(\subspaceU_{2}\), respectively.
Suppose \(\U_{1}\) and \(\U_{2}\) are orthonormal bases 
for \(\subspaceU_{1}\) and \(\subspaceU_{2}\), respectively.
Then, the distance between the two subspaces can also be expressed as
\(\dist(\subspaceU_{1}, \subspaceU_{2}) 
= \norm{(\I - \U_{1}\U_{1}\trans)\U_{2}}_{2}
= \norm{(\I - \U_{2}\U_{2}\trans)\U_{1}}_{2}\).

\subsection{Randomized Sketching} \label{subsec:sketching}

Let \(\sksubspace \subseteq \bbR^{\n}\) be a subspace.
A sketching matrix \(\skmat \in \bbR^{\skdim \times \n}\)
is an \(\skeps\)-subspace embedding for \(\sksubspace\) if
\begin{equation}\label{eq:subspace_embedding}
(1 - \skeps) \norm{\skvec}_{2}^{2}
\leq
\norm{\skmat \skvec}_{2}^{2}
\leq
(1 + \skeps) \norm{\skvec}_{2}^{2},
\qquad
\skvec \in \sksubspace,
\end{equation}
where \(0 < \skeps < 1\) is the distortion parameter and \(\skdim\)
is the sketch dimension, with \(\dim(\sksubspace) \leq \skdim \ll \n\).
Typically, \(\skmat\) is chosen as a random matrix such that~\eqref{eq:subspace_embedding} holds with high probability,
and the application of \(\skmat\) to a vector can be computed efficiently.
For example, a sparse sign matrix with a fixed number of nonzero entries
per column~\cite{Martinsson_Tropp_2020} allows this application to be computed
in \(\bigO(\n)\) operations.

\subsection{Basis Whitening} \label{subsec:basis_whitening}

Basis whitening uses a sketching matrix
to construct a well-conditioned basis for a given subspace~\cite{Nakatsukasa_Tropp_2024}.
Suppose \(\X \in \bbR^{\n \times \subspacedim}\) has full column rank.
If \(\skmat \in \bbR^{\skdim \times \n}\) is an \(\skeps\)-subspace embedding for \(\Span\{\X\}\),
then \(\skX = \skmat \X\) has full column rank.
Let
\begin{equation*}
\skX
=
\skwhX \whiteingR
\end{equation*}
be the QR factorization of \(\skX\),
where \(\skwhX \in \bbR^{\skdim \times \subspacedim}\) is orthonormal and \(\whiteingR \in \bbR^{\subspacedim \times \subspacedim}\) is nonsingular and upper triangular.
We denote this process as a subroutine \([\skwhX, \whiteingR] = \QRfunc(\skX)\).
The whitened basis
\begin{equation*}
\whX = \X \whiteingR^{-1}
\end{equation*}
is therefore sketch-orthonormal.
The nonsingularity of \(\whiteingR\) implies
\(\Span\{\whX\}= \Span\{\X\}\).
The subspace embedding property~\eqref{eq:subspace_embedding} also gives
\begin{equation*}
\cond_{2}(\whX)
\leq \sqrt{\frac{1 + \skeps}{1 - \skeps}},
\end{equation*}
meaning that \(\whX\) is well-conditioned if \(\skeps\) is not too close to \(1\).

For \(\n \gg \subspacedim\), computing a QR factorization of \(\X\) has leading cost
\(\constQR \n \subspacedim^{2}\) operations.
Here, \(\constQR = 2\) for CGS and MGS,
and \(\constQR = 4\) for CGS2~(CGS with one reorthogonalization)
and Householder QR with the orthonormal factor formed explicitly.
Whitening requires three steps: 
applying \(\skmat\) to \(\X\),
computing the QR factorization of \(\skX\), 
and computing
\(\whX = \X \whiteingR^{-1}\) by a triangular solve.
For sparse sign matrices in Section~\ref{subsec:sketching},
the first step costs \(\bigO(\n \subspacedim)\) operations.
The QR factorization of \(\skX\) costs \(\bigO(\skdim \subspacedim^{2})\) operations,
and the triangular solve has leading cost \(\n \subspacedim^{2}\).
When the first two steps contribute only lower-order costs,
the leading cost of whitening
is therefore approximately \(\n \subspacedim^{2}\),
which means that whitening reduces the leading cost by a factor of \(2\)--\(4\).

\section{Sketched Subspace Iteration and TraceMIN}
\label{sec:sketched_eigensolvers}

We combine whitening with a real Schur formulation of sRR to construct two eigensolvers: sSI and sTraceMIN.
Throughout this section, we seek the \(\nev\) smallest eigenpairs of an SPD matrix \(\A \in \bbR^{\n \times \n}\).
The initial guess is a full-column-rank matrix \(\X \in \bbR^{\n \times \nex}\), where \(\nev < \nex \ll \n\).
Each method uses a fixed sketching matrix \(\skmat\), and all sketched bases to be whitened are assumed to have full column rank.

\subsection{Sketched Rayleigh--Ritz via Real Schur Decomposition} \label{subsec:sRR}

For a search basis \(\X \in \bbR^{\n \times \nex}\), the RR procedure solves the projected eigenvalue problem \(\Ap \Z = \Z \Th\),
where \(\Ap = \X\pinv \A \X\) is the solution of the least-squares problem
\begin{equation*}
  \min_{\M \in \bbR^{\nex \times \nex}} \norm{\A \X - \X \M}_{\fro}.
\end{equation*}
sRR replaces the least-squares problem by its sketched version~\cite[section~6.3]{Nakatsukasa_Tropp_2024}
\begin{equation*}
  \min_{\M \in \bbR^{\nex \times \nex}} \norm{\skmat (\A \X - \X \M)}_{\fro},
\end{equation*}
whose solution is \(\Apsk=(\skmat \X)\pinv (\skmat \A \X)\).
In the remainder of this subsection, we assume that \(\X\) is sketch-orthonormal.
The sketched projected matrix then simplifies to
\begin{equation*}
    \Apsk = (\skmat \X)\trans (\skmat \A \X).
\end{equation*}
Once \(\skmat \X\) and \(\skmat \A \X\) are available, forming \(\Apsk\) costs \(\bigO(\skdim \nex^{2})\) operations.
If \(\skmat\) satisfies~\eqref{eq:subspace_embedding} on \(\Span\{\X, \A \X\}\), 
then the sketched least-squares minimizer satisfies
\begin{equation*}
\norm{\A \X - \X \Apsk}_{\fro}
\leq
\sqrt{\frac{1 + \skeps}{1 - \skeps}}
\min_{\M \in \bbR^{\nex \times \nex}}
\norm{\A \X - \X \M}_{\fro}.
\end{equation*}

The sketched projected matrix \(\Apsk\) is generally nonsymmetric
and may have nonreal eigenpairs 
even though \(\A\) is SPD~\cite[section~6.8]{Nakatsukasa_Tropp_2024}.
To keep the algorithms in real arithmetic, we instead compute a real Schur decomposition
\begin{equation*}
    \Apsk = \Z \T \Z\trans,
\end{equation*}
where \(\Z\) is orthogonal and  \(\T\) is upper quasi-triangular.
Then, the transformed basis \(\X \Z\) remains real and sketch-orthonormal,
and the pair \((\X \Z, \T)\) represents an approximate characteristic subspace.
Algorithm~\ref{alg:sRR} summarizes this extraction procedure.
See also Algorithm~\ref{alg:RR} in Appendix~\ref{app:RR_SI_TraceMIN} 
for the standard RR procedure for comparison.

The columns of \(\Z\) are Schur vectors of \(\Apsk\) 
and need not be its eigenvectors.
Thus, the columns of \(\X \Z\) need not 
individually approximate eigenvectors of \(\A\).
This distinction affects how Algorithm~\ref{alg:sRR} is used in eigensolvers.
Since \(\nex > \nev\),
an additional selection step must be performed
to determine the approximate characteristic subspace
with respect to the desired part of the spectrum.
In standard RR, we select the \(\nev\) Ritz pairs 
corresponding to the smallest Ritz values.
These approximations are used to monitor convergence.
For a real Schur decomposition,
we reorder the diagonal blocks of \(\T\) 
by increasing real parts of their eigenvalues~\cite{Kressner_2006}.
Selecting complete leading blocks of \(\T\) 
and the corresponding columns of \(\X \Z\) 
gives a real basis of the desired approximate characteristic subspace.
The implementation details are discussed in Section~\ref{sec:implementation}.

\begin{algorithm}[ht!]
  \caption{Sketched Rayleigh--Ritz procedure with real Schur decomposition} \label{alg:sRR}
  \begin{algorithmic}[1]
    \FUNCTION{\([\Z, \T] = \sRRfunc(\skX, \skAX)\)}
    \REQUIRE{\(\skX = \skmat \X\) and \(\skAX = \skmat \A \X\) for a sketch-orthonormal basis \(\X\)}
    \ENSURE{Orthogonal matrix \(\Z\) and upper quasi-triangular matrix \(\T\)}
    \STATE{\(\Apsk = \skX\trans \skAX\)}
    \STATE{Compute a real Schur decomposition \(\Apsk = \Z \T \Z\trans\)}
    \STATE{Reorder the diagonal blocks of \(\T\) by increasing real parts of their eigenvalues and update \(\Z\) accordingly}
  \end{algorithmic}
\end{algorithm}

\subsection{Sketched Subspace Iteration} \label{subsec:sSI}

We begin with the inverse-iteration variant of the classical SI algorithm~\cite{Demmel_1997}.
At each iteration, SI 
updates the current basis by \(\X \leftarrow \A^{-1} \X\), 
orthogonalizes \(\X\), 
and then applies the RR procedure to extract the approximate eigenpairs.
The standard SI algorithm is summarized in Algorithm~\ref{alg:SI} in Appendix~\ref{app:RR_SI_TraceMIN}.
Note that \(\A\) only needs to be factored once at the beginning of the iteration,
and the factorization is reused in all subsequent iterations when \(\A^{-1}\) is applied.

sSI modifies three key components of SI:
the orthogonalization of \(\X\), 
the RR procedure on the orthonormal basis \(\X\),
and the convergence check based on the residual \(\R = \A \X - \X \Th\),
where \(\Th\) is the diagonal matrix of Ritz values corresponding to \(\X\).
These are replaced, respectively, by 
the whitening of \(\X\),
the sRR procedure on the sketch-orthonormal basis \(\X\),
and the convergence check based on the sketched residual \(\skR = \skmat \A \X - (\skmat \X)\T\);
see Steps~\ref{step:sSI_whitening_iter}, \ref{step:sSI_sRR_iter},
and~\ref{step:sSI_skRused_iter} of Algorithm~\ref{alg:sSI}.
Moreover, the sketches \(\skX=\skmat\X\) and \(\skAX=\skmat\A\X\) can be stored
and updated in each iteration cheaply.
The two updates of \(\X\) arising from whitening and subspace extraction
can be combined into a single update.
Implementation details, including the definition of the term \(\mynev\), are discussed in
Section~\ref{sec:implementation}.

For symmetric EVPs, users typically seek a set of orthonormal eigenvectors.
Therefore, at the termination of the algorithm, 
we perform a standard QR factorization
and RR procedure to obtain the final eigenpairs for output.

\begin{algorithm}[ht!]
  \caption{Sketched subspace iteration} \label{alg:sSI}
  \begin{algorithmic}[1]
    \FUNCTION{\([\V, \Lam] = \sSIfunc(\skmat, \A, \X)\)}
    \REQUIRE{Sketching matrix \(\skmat\), SPD matrix \(\A\), initial guess \(\X\)}
    \ENSURE{Approximate eigenpairs \((\V, \Lam)\)}
    \STATE{\(\skX = \skmat \X\), \(\skAX = \skmat \A \X\)}
    \STATE{\([\skX, \whiteingR] = \QRfunc(\skX)\), \(\skAX = \skAX \whiteingR^{-1}\)} \label{step:sSI_whitening0}
    \STATE{\([\Z, \T] = \sRRfunc(\skX, \skAX)\) and determine \(\mynev\) from \(\T\)} \label{step:sSI_sRR0}
    \STATE{\(\X = \X (\whiteingR^{-1} \Z)\), \(\skX = \skX \Z\), \(\skAX = \skAX \Z\)} \label{step:sSI_poststep0}
    \STATE{\(\skRused = \skAX(:, 1 : \mynev) - \skX(:, 1 : \mynev) \T(1 : \mynev, 1 : \mynev)\)} \label{step:sSI_skRused0}
    \STATE{Check convergence using \(\skRused\)}
    \WHILE{not converged}
      \STATE{\(\X \gets \A^{-1} \X\), \(\skAX = \skX\), \(\skX = \skmat \X\)} \label{step:sSI_inverse_iter}
      \STATE{\([\skX, \whiteingR] = \QRfunc(\skX)\), \(\skAX = \skAX \whiteingR^{-1}\)} \label{step:sSI_whitening_iter}
      \STATE{\([\Z, \T] = \sRRfunc(\skX, \skAX)\) and determine \(\mynev\) from \(\T\)} \label{step:sSI_sRR_iter}
      \STATE{\(\X = \X (\whiteingR^{-1} \Z)\), \(\skX = \skX \Z\), \(\skAX = \skAX \Z\)} \label{step:sSI_poststep_iter}
      \STATE{\(\skRused = \skAX(:, 1 : \mynev) - \skX(:, 1 : \mynev) \T(1 : \mynev, 1 : \mynev)\)} \label{step:sSI_skRused_iter}
      \STATE{Check convergence using \(\skRused\)}
    \ENDWHILE
    \STATE{\([\X, \sim] = \QRfunc\bigl(\X(:, 1 : \mynev)\bigr)\)}
    \STATE{\([\Z, \Th] = \RRfunc(\A, \X)\)}
    \STATE{\(\V = \X \Z(:, 1 : \nev)\), \(\Lam = \Th(1 : \nev, 1 : \nev)\)}
  \end{algorithmic}
\end{algorithm}

\subsection{Sketched TraceMIN} \label{subsec:sTraceMIN}

TraceMIN~\cite{Sameh_Wisniewski_1982} can be viewed as a variant of SI.
Instead of updating the search space by
\(\X \leftarrow \A^{-1}\X\), it computes a correction \(\TrMINDel\) at each
iteration and updates \(\X \leftarrow \X-\TrMINDel\).
The correction is determined by the following constrained trace-minimization
problem:
\begin{equation}
\label{eq:traceminoptmization}
\min_{\substack{\TrMINDel\in\bbR^{\n\times\nex}\\\X\trans \TrMINDel = \zeromat}}
\trace\bigl((\X - \TrMINDel)\trans \A (\X - \TrMINDel)\bigr).
\end{equation}

Let \(\X\) be an orthonormal basis and \(\TrMINP = \I_{\n} - \X \X\trans\) be the orthogonal projector onto \(\Null(\X\trans)\).
Then, \eqref{eq:traceminoptmization} is equivalent to solving
\begin{equation} \label{eq:TraceMIN_correction}
\TrMINP \A \TrMINP \TrMINDel = \TrMINP \R,
\qquad
\TrMINP \TrMINDel = \TrMINDel,
\end{equation}
where \(\R = \A \X - \X \Th\).
In practice, the linear systems~\eqref{eq:TraceMIN_correction} are usually
solved by inexact linear solvers such as GMRES~\cite{Klinvex_Saied_Sameh_2013}.
We then orthogonalize the updated basis and apply the RR procedure.
The standard TraceMIN algorithm is summarized in Algorithm~\ref{alg:TraceMIN} in Appendix~\ref{app:RR_SI_TraceMIN}.

The above discussion suggests constructing sTraceMIN by replacing the inverse-iteration step in sSI with a TraceMIN correction step based on~\eqref{eq:TraceMIN_correction}.
The main difference is that the basis \(\X\) in sSI is sketch-orthonormal.
We therefore replace the Euclidean projection in~\eqref{eq:TraceMIN_correction} by its sketched version, leading to
\begin{equation} \label{eq:sTraceMIN_correction}
\TrMINPsk \A \TrMINPsk \TrMINDel
= \TrMINPsk \R,
\qquad
\TrMINPsk\TrMINDel=\TrMINDel,
\end{equation}
where \(\R = \A \X - \X \T\) and \(\TrMINPsk= \I_{\n} - \X (\skmat\X) \trans \skmat\).
Algorithm~\ref{alg:sTraceMIN} gives the pseudocode of sTraceMIN.
Similarly, we perform a standard QR factorization
and RR procedure at termination 
to obtain the final returned eigenpairs.

\begin{remark}[On the solvability of~\eqref{eq:TraceMIN_correction} and~\eqref{eq:sTraceMIN_correction}]
When \(\A\) is positive definite,~\eqref{eq:TraceMIN_correction} always has a unique solution.
In contrast, \eqref{eq:sTraceMIN_correction} can have no solution.
However, we can prove that if \(\skmat\) is an \(\skeps\)-subspace embedding 
for the characteristic subspace associated with the \(\nex\) smallest eigenvalues,
and \(\Span\{\X\}\) is sufficiently close to this subspace,
then~\eqref{eq:sTraceMIN_correction} is guaranteed to be solvable.
A detailed discussion is given in Section~\ref{subsec:local_equiv}.
\end{remark}

\begin{algorithm}[th!]
  \caption{Sketched TraceMIN} \label{alg:sTraceMIN}
  \begin{algorithmic}[1]
    \FUNCTION{\([\V, \Lam] = \sTrMINfunc(\skmat, \A, \X)\)}
    \REQUIRE{Sketching matrix \(\skmat\), SPD matrix \(\A\), initial guess \(\X\)}
    \ENSURE{Approximate eigenpairs \((\V, \Lam)\)}
    \STATE{\(\skX = \skmat \X\), \(\skAX= \skmat \A \X\)}
    \STATE{\([\skX, \whiteingR] = \QRfunc(\skX)\), \(\skAX = \skAX \whiteingR^{-1}\)}
    \STATE{\([\Z, \T] = \sRRfunc(\skX, \skAX)\) and determine \(\mynev\) from \(\T\)}
    \STATE{\(\X = \X (\whiteingR^{-1} \Z)\), \(\skX = \skX \Z\), \(\skAX = \skAX \Z\)}
    \STATE{\(\skRused = \skAX(:, 1 : \mynev) - \skX(:, 1 : \mynev) \T(1 : \mynev, 1 : \mynev)\)}
    \STATE{Check convergence using \(\skRused\)}
    \WHILE{not converged}
      \STATE{Construct the projector \(\TrMINPsk = \I_{\n} - \X \skX\trans \skmat\)}
      \STATE{Approximately solve \(\TrMINPsk \A \TrMINPsk \TrMINDel = \TrMINPsk \A \X\) and update \(\TrMINDel = \TrMINPsk \TrMINDel\)}
      \label{step:sTraceMIN_correctionequation}
      \STATE{\(\X= \X - \TrMINDel\)}
      \STATE{\(\skX = \skmat \X\), \(\skAX = \skmat \A \X\)}
      \label{step:sTraceMIN_getskXskAX}
      \STATE{\([\skX, \whiteingR] = \QRfunc(\skX)\), \(\skAX = \skAX \whiteingR^{-1}\)}
      \STATE{\([\Z, \T] = \sRRfunc(\skX, \skAX)\) and determine \(\mynev\) from \(\T\)}
      \STATE{\(\X = \X (\whiteingR^{-1} \Z)\), \(\skX = \skX \Z\), \(\skAX = \skAX \Z\)}
      \label{step:sTraceMIN_poststep_iter}
      \STATE{\(\skRused = \skAX(:, 1 : \mynev) - \skX(:, 1 : \mynev) \T(1 : \mynev, 1 : \mynev)\)}
      \STATE{Check convergence using \(\skRused\)}
    \ENDWHILE
    \STATE{\([\X, \sim] = \QRfunc(\X(:, 1 : \mynev))\)}
    \STATE{\([\Z, \Th] = \RRfunc(\A, \X)\)}
    \STATE{\(\V = \X \Z(:, 1 : \nev)\), \(\Lam = \Th(1 : \nev, 1 : \nev)\)}
  \end{algorithmic}
\end{algorithm}

\section{Implementation Details} \label{sec:implementation}

We now turn to the implementation details in the sSI and sTraceMIN algorithms, 
including Schur block selection, convergence checks and consistent updates of the basis and its sketches.

\subsection{Schur Block Reorder and Selection} \label{subsec:reorder}

As part of each sRR procedure~(Algorithm~\ref{alg:sRR}), we reorder the diagonal blocks of \(\T\)
in order of increasing real parts of their eigenvalues by solving several Sylvester
equations~\cite{Kressner_2006}, and update \(\Z\) accordingly.
After updating \(\X \gets \X (\whiteingR^{-1} \Z)\),
we select the first \(\nev\) columns of \(\X\)
to be the basis of the approximate characteristic subspace.
If this selection splits a \(2 \times 2\) diagonal block of \(\T\),
we enlarge the selected subspace by one column.
More precisely, we choose
\begin{equation*}
\mynev =
\begin{cases}
    \nev + 1, & \T(\nev + 1, \nev) \neq 0, \\
    \nev, & \text{otherwise}.
\end{cases}
\end{equation*}
This choice ensures that the selected columns do not break a \(2 \times 2\) block.

\subsection{Residual and Convergence Check} \label{subsec:residual}

For the standard SI and TraceMIN methods, we consider
the block residual in the form
\begin{equation*}
\Rused
= \A \X(:, 1 : \nev)
- \X(:, 1 : \nev) \Th(1 : \nev, 1 : \nev),
\end{equation*}
and use the following relative residual for convergence checking:
\begin{equation}
\label{eq:rel_res_entry}
\resnorm
= \frac{\norm{\Rused}_{\fro}}
{\norm{\X(:, 1 : \nev)}_{\fro}
\bigl(\norm{\A}_2 + \norm{\Th(1 : \nev, 1 : \nev)}_2\bigr)}.
\end{equation}

Similarly, the block residual for sSI and sTraceMIN is defined as
\begin{equation*}
\Rused
= \A \X(:, 1 : \mynev)
- \X(:, 1 : \mynev) \T(1 : \mynev, 1 : \mynev).
\end{equation*}
However, since \(\T\) is upper quasi-triangular, the multiplication
\(\X(:, 1 : \mynev) \T(1 : \mynev, 1 : \mynev)\)
can take \(\bigO(\n \nev^{2})\) operations,
leading to a nonnegligible computational cost.
Therefore, we define the sketched block residual
\begin{equation}
\label{eq:sresidual_used}
\skRused
= \skAX(:, 1 : \mynev)
- \skX(:, 1 : \mynev) \T(1 : \mynev, 1 : \mynev)
\end{equation}
and use the following sketched relative residual for convergence checks instead:
\begin{equation}
\label{eq:rel_sres_entry}
\skresnorm
= \frac{\norm{\skRused}_{\fro}}
{\norm{\skX(:, 1 : \mynev)}_{\fro}
\bigl(\norm{\A}_2 + \norm{\T(1 : \mynev, 1 : \mynev)}_2\bigr)}.
\end{equation}
Since \(\skX\) and \(\skAX\) are both available in sSI and sTraceMIN,
computing~\eqref{eq:sresidual_used} takes only
\(\bigO(\skdim \nev^{2})\) operations.

In practice, for a given tolerance \(\tol\), the algorithms halt if
\(\resnorm \leq \tol\) for SI and TraceMIN or if \(\skresnorm \leq \tol\) for
sSI and sTraceMIN.
The norm \(\norm{\A}_2\) is estimated by the MATLAB function \(\normestfunc(\A, 10^{-3})\)
at the beginning of each algorithm.

\begin{remark}
It is worth noting that, when using~\eqref{eq:rel_sres_entry}
for convergence checks, neither sSI nor sTraceMIN needs the
original-space Schur residual \(\R\) during the iteration.
For sSI, this is immediate.
For sTraceMIN, the right-hand side of equation~\eqref{eq:sTraceMIN_correction}
satisfies \(\TrMINPsk \R = \TrMINPsk \A \X\),
which allows us to form the right-hand side directly from \(\A\X\).
This avoids the additional matrix product \(\X \T\) 
required to construct \(\R\).
When \(\A\) is sufficiently sparse, computing \(\X\T\)
can be substantially more expensive than computing \(\A\X\);
see also the discussion in Section~\ref{subsec:basis_orthonormalization_updates}.
\end{remark}

\subsection{Basis Orthonormalization and Updates} \label{subsec:basis_orthonormalization_updates}

In Step~\ref{step:sSI_poststep_iter} of Algorithm~\ref{alg:sSI}
and Step~\ref{step:sTraceMIN_poststep_iter} of
Algorithm~\ref{alg:sTraceMIN}, the basis update
\(\X \leftarrow \X(\whiteingR^{-1}\Z)\)
combines the transformations associated with \QRfunc{} (\(\X \leftarrow \X\whiteingR^{-1}\)) and
\sRRfunc{} (\(\X \leftarrow \X\Z\)).
The whitened basis
\(\X \whiteingR^{-1}\) need not be formed since \sRRfunc{} requires only \(\skX\) and \(\skAX\),
Instead, we first compute \(\whiteingR^{-1} \Z\) by a triangular
solve and then apply the combined transformation to \(\X\).
This reduces the cost of the two basis transformations from
approximately \(3 \n \nex^{2}\) to \(2 \n \nex^{2} + \nex^{3}\)
operations, saving approximately \(\n \nex^{2}\) operations when
\(\nex \ll \n\).

The sketches \(\skX\) and \(\skAX\), however, are updated twice after whitening and sRR.
Nevertheless, each update requires only \(\bigO(\skdim \nex^{2})\)
operations and is therefore inexpensive compared with the
original-space basis update when \(\skdim \ll \n\).

For sTraceMIN, another possible choice is to store
\(\AX = \A \X\) in Step~\ref{step:sTraceMIN_getskXskAX} of
Algorithm~\ref{alg:sTraceMIN}, and update it alongside \(\X\) in
Step~\ref{step:sTraceMIN_poststep_iter} by
\(\AX \leftarrow \AX(\whiteingR^{-1} \Z)\).
The maintained \(\AX\) can then be used in the correction equation
in Step~\ref{step:sTraceMIN_correctionequation} of the next iteration,
avoiding another application of \(\A\) to \(\X\).
However, the update \(\AX \leftarrow \AX(\whiteingR^{-1} \Z)\) takes
\(\bigO(\n \nex^{2})\) operations.
On the other hand, for sparse \(\A\), computing \(\AX = \A \X\) costs
\(\bigO(\nnz(\A) \nex)\) operations, where \(\nnz(\A)\)
is the number of nonzero entries of \(\A\).
In practice, the choice of maintaining \(\AX\) or recomputing it depends on the
implementation of sparse matrix and the hardware.
On our test platform, recomputing \(\AX\) is faster,
and, accordingly, our implementation of sTraceMIN does not maintain \(\AX\).

Together with the sketched convergence checks in
Section~\ref{subsec:residual}, these implementation choices
give a leading cost of \(2 \n \nex^{2}\) operations per iteration
for sSI and sTraceMIN, compared with \((\constQR + 3) \n \nex^{2}\)
for SI and TraceMIN when the symmetry of
\(\X\trans\A\X\) is exploited.%
\footnote{
Our implementations by MATLAB of fail to exploit this
symmetry, resulting in a leading cost of
\((\constQR + 4) \n \nex^{2}\) operations per iteration.
}
Here, \(\constQR\) is the QR cost coefficient defined in
Section~\ref{subsec:basis_whitening}.
Note that these estimates assume no soft locking is applied and exclude the costs of
applying \(\A\) and \(\skmat\), linear solves, and lower-order
computations in the sketched and reduced spaces.

\section{Convergence Analysis} \label{sec:convergence}

The main purpose of this section is to establish the local convergence of sSI and sTraceMIN.
The analysis proceeds as follows.
\begin{enumerate}
\item In Section~\ref{subsec:local_behavior}, 
we discuss the suitable subspace embedding assumptions 
for the proposed eigensolvers
when the sketching matrix \(\skmat\) is fixed throughout all iterations,
and propose a local embedding assumption for \(\skmat\).
With its help, we prove
the well-definedness of basis whitening~(Lemma~\ref{lem:nearby_subspace_embedding} and Corollary~\ref{cor:local_uniformity_subspace_embedding}).
\item Under the local embedding assumption, in Section~\ref{subsec:local_equiv},
we prove that the search subspaces generated by SI and sSI~(Theorem~\ref{thm:local_equiv_SI_sSI}),
as well as the ones generated by TraceMIN and sTraceMIN 
(when the correction equations are solved exactly,
Theorem~\ref{thm:local_equiv_TraceMIN_sTraceMIN}),
are the same when close to convergence.
\item Finally, in Section~\ref{subsec:local_closeness_RR_sRR},
we provide a bound on the gap between the subspaces selected by
the reordered spectral decomposition of \(\Ap\)
and the reordered real Schur decomposition of \(\Apsk\)
(Theorem~\ref{thm:local_gap_RR_sRR}).
\end{enumerate}
Combining the local equivalence of the search subspaces and
the gap between the extracted subspaces,
we obtain a full picture of the local convergence of sSI and sTraceMIN.

Throughout this section, \(\A\) is assumed to be SPD and its spectral decomposition is 
\begin{equation}\label{eq:A_spectral_decomposition}
\A
[\V_{1}, \V_{2}]
=
[\V_{1}, \V_{2}]
\diag\{\Lam_{1}, \Lam_{2}\},
\end{equation}
where \(\Lam_{1} = \diag\{\lam_{1}, \ldots, \lam_{\nex}\}\)
and \(\Lam_{2} = \diag\{\lam_{\nex + 1}, \ldots, \lam_{\n}\}\).
The eigenvalues of \(\A\) satisfy \(0 < \lam_{1} \leq \cdots \leq \lam_{\n}\),
and the columns of \(\V_{1} \in \bbR^{\n \times \nex}\) and \(\V_{2} \in \bbR^{\n \times (\n - \nex)}\) are the corresponding eigenvectors.
For simplicity, we assume that \(\nev < \nex\), 
\(\lam_{\nev} < \lam_{\nev + 1}\) and \(\lam_{\nex} < \lam_{\nex + 1}\).
Throughout this section,
we use \(\X\) and \(\XS\) to denote, respectively, 
an orthonormal basis and a sketch-orthonormal basis for the same search subspace, satisfying
\begin{equation} \label{eq:search_basis_relations}
\X\trans \X = \I_{\nex},
\quad
(\skmat \XS)\trans (\skmat\XS) = \I_{\nex},
\quad
\XS = \X \C,
\quad \C \in \bbR^{\nex\times\nex}.
\end{equation}
We also assume that the sketching dimension \(\skdim\) satisfies \(\nex \leq \skdim \ll \n\).
Correspondingly, the residuals of the search subspaces using RR and sRR are defined as
\begin{equation}\label{eq:search_residuals}
\begin{gathered}
\R = \A\X - \X \Ap,
\qquad
\Ap = \X\trans \A \X,
\\
\RS = \A \XS - \XS \Apsk,
\qquad
\Apsk = (\skmat \XS)\trans (\skmat \A \XS).
\end{gathered}
\end{equation}
We work in exact arithmetic and assume that soft lock is not applied.
We also assume that all linear systems, including the inverse-iteration
steps in SI and sSI and the correction equations
\eqref{eq:TraceMIN_correction} in TraceMIN and
\eqref{eq:sTraceMIN_correction} in sTraceMIN, 
are solved exactly.

\subsection{Subspace Embedding Assumptions for Eigensolvers} \label{subsec:local_behavior}

As discussed in Section~\ref{subsec:basis_whitening}, if \(\skmat\) is an
\(\skeps\)-subspace embedding for \(\Span\{\X\}\), then \(\skmat \X\)
has full column rank.
Consequently, the whitened basis spans the same subspace as the
orthonormal basis obtained by exact QR factorization.
Analyses of sketch-based iterative eigensolvers often assume that
a fixed \(\skmat\) is an \(\skeps\)-subspace embedding
for the search subspace at every iteration;
see, for example~\cite[Theorem 4.12]{deDamas_Grigori_2025}.
If this holds, SI and sSI generate exactly the same search subspace
at each iteration, whenever they start from the same initial guess.
Convergence of these search subspaces then follows from
the classical analysis of SI~\cite{Demmel_1997}.
However, this assumption is difficult to guarantee throughout all the
iterations of an eigensolver, since the search subspace may explore
arbitrary directions in \(\bbR^{\n}\) before converging to \(\Span\{\V_{1}\}\).

Nevertheless, this assumption has a local justification.
The following lemma and corollary show that,
if \(\skmat\) is an \(\skeps\)-subspace embedding 
for \(\Span\{\V_{1}\}\), 
then it is also a subspace embedding,
with a controlled distortion, 
for every \(\Span\{\X\}\) with \(\dist\bigl(\Span\{\X\},\Span\{\V_{1}\}\bigr)\) sufficiently small.

\begin{lemma}[Embedding of a nearby subspace]\label{lem:nearby_subspace_embedding}
Let \(\V_{1} \in \bbR^{\n \times \nex}\) be an orthonormal matrix
and let \(\skmat\) be an \(\skeps\)-subspace embedding for \(\Span\{\V_{1}\}\).
Define
\(\sklower=\sqrt{1-\skeps}\),
\(\skupper=\sqrt{1+\skeps}\),
and \(\skspnorm=\norm{\skmat}_{2}\).
Let \(\ubE \geq 0\) be a fixed scalar satisfying \(\skspnorm\ubE<\sklower\).%
\footnote{The subspace embedding property of \(\skmat\) gives \(\sklower \leq \skspnorm\).
Therefore, \(\ubE < 1\) when it satisfies the condition.}
For every orthonormal \(\X \in \bbR^{\n \times \nex}\), if
\begin{equation*}
\dist\bigl(\Span\{\X\}, \Span\{\V_{1}\}\bigr) \leq \ubE,
\end{equation*}
then \(\skmat\X\) has full column rank, and
\begin{equation}\label{eq:sketched_basis_singular_values}
\sv_{\min}(\skmat\X)
\geq
\sklower - \skspnorm\ubE,
\qquad
\sv_{\max}(\skmat\X)
\leq
\skupper + \skspnorm\ubE.
\end{equation}
Moreover,
\begin{equation}\label{eq:nearby_subspace_embedding}
(\sklower - \skspnorm \ubE) \norm{\x}_{2}
\leq
\norm{\skmat \x}_{2}
\leq
(\skupper + \skspnorm \ubE)\norm{\x}_{2},
\qquad
\x\in\Span\{\X\}.
\end{equation}
\end{lemma}

\begin{proof}
Let \(\P_{\X} = \X \X\trans\) 
and \(\P_{\V_{1}} = \V_{1}\V_{1}\trans\).
Then
\begin{equation*}
\norm{\skmat\P_{\X}
-\skmat\P_{\V_{1}}}_{2}
\leq
\skspnorm \norm{\P_{\X}-\P_{\V_{1}}}_{2}
=
\skspnorm \dist\bigl(\Span\{\X\}, \Span\{\V_{1}\}\bigr)
\leq
\skspnorm\ubE.
\end{equation*}
Thus, the singular value perturbation inequality yields
\begin{equation*}
\abs{\sv_{j}(\skmat\X)-\sv_{j}(\skmat\V_{1})}
=\abs{\sv_{j}(\skmat\P_{\X})-\sv_{j}(\skmat\P_{\V_{1}})}
\leq\norm{\skmat\P_{\X}-\skmat\P_{\V_{1}}}_{2}
\leq\skspnorm \ubE,
\end{equation*}
for \(1 \leq j \leq \nex\).
The subspace embedding property~\eqref{eq:subspace_embedding}
and the orthonormality of \(\V_{1}\) imply
\begin{equation*}
\sv_{\min}(\skmat\V_{1})\geq\sklower,
\qquad
\sv_{\max}(\skmat\V_{1})\leq\skupper.
\end{equation*}
Combining these inequalities proves
\eqref{eq:sketched_basis_singular_values}.
Since \(\sklower-\skspnorm\ubE>0\),
the matrix \(\skmat\X\) has full column rank.
Finally, every \(\x\in\Span\{\X\}\) can be written as
\(\x=\X\c\), with
\(\norm{\x}_{2}=\norm{\c}_{2}\).
Applying~\eqref{eq:sketched_basis_singular_values}
to \(\skmat\x=\skmat\X\c\) gives
\eqref{eq:nearby_subspace_embedding}.
\end{proof}

\begin{corollary}[Local uniformity of the subspace embedding]
\label{cor:local_uniformity_subspace_embedding}
Using the notation of Lemma~\ref{lem:nearby_subspace_embedding},
fix \(\skeps<\skepsneighbor<1\) and define
\begin{equation*}
\ubEneighbor
=
\frac{1}{\skspnorm}
\min\left\{
\sklower-\sqrt{1-\skepsneighbor},\,
\sqrt{1+\skepsneighbor}-\skupper
\right\}
\in [0, 1)
\end{equation*}
Then \(\skmat\) is an \(\skepsneighbor\)-subspace embedding for
\(\Span\{\X\}\) for every orthonormal
\(\X\in\bbR^{\n\times\nex}\) satisfying
\begin{equation*}
\dist\bigl(\Span\{\X\},\Span\{\V_{1}\}\bigr) \leq \ubEneighbor.
\end{equation*}
In particular, basis whitening is well-defined for every such
\(\X\) and preserves its column space, as does exact QR factorization.
\end{corollary}

\begin{proof}
The definition of \(\ubEneighbor\) gives
\begin{equation*}
\sklower-\skspnorm\ubEneighbor
\geq\sqrt{1-\skepsneighbor}>0,
\qquad
\skupper+\skspnorm\ubEneighbor
\leq\sqrt{1+\skepsneighbor}.
\end{equation*}
The result follows from Lemma~\ref{lem:nearby_subspace_embedding}
by squaring the bounds in~\eqref{eq:nearby_subspace_embedding}.
\end{proof}

The analysis of sRR typically assumes that \(\skmat\) is a
subspace embedding for \(\Span\{\X, \A\X\}\);
see~\cite[Section~6.3]{Nakatsukasa_Tropp_2024}.
For iterative sketched eigensolvers, this condition may be
imposed at every iteration, as in~\cite[Theorem~4.7]{Grigori_Park_Simunec_2026}.
Although the joint subspace has at most \(\nex\) additional
dimensions, its embedding property cannot be guaranteed
uniformly over a neighborhood of \(\V_{1}\).
\begin{proposition}
\label{prop:no_uniform_embedding_X_AX}
Using the notation of~\eqref{eq:A_spectral_decomposition}.
Let \(\skmat \in \bbR^{\skdim \times \n}\) with \(\nex \leq \skdim < \n\).
For every \(\ubE \in (0, 1]\), there exists an orthonormal
\(\X \in \bbR^{\n \times \nex}\) satisfying
\begin{equation*}
\dist\bigl(\Span\{\X\},\Span\{\V_{1}\}\bigr) \leq \ubE
\end{equation*}
such that \(\skmat\) is not an \(\skeps\)-subspace embedding
for \(\Span\{\X,\A\X\}\) for any \(\skeps \in [0,1)\).
\end{proposition}

\begin{proof}
Choose a unit vector \(\z \in \Null(\skmat)\), and set
\(i=\nev\) and \(j=\nev+1\), so that \(\lam_i<\lam_j\).
Let
\begin{equation*}
\U=\V_{1}+t\z(\unitvec_{i}-\unitvec_{j})\trans
\end{equation*}
for some \(0 < t < \ubE/ \sqrt{2}\),
where \(\unitvec_{i},\unitvec_{j}\) are standard basis vectors
in \(\bbR^{\nex}\).
Since \(\norm{\U-\V_{1}}_{2}=\sqrt{2}t < \ubE\),
the matrix \(\U\) has full column rank.
Let \(\X\) be an orthonormal basis for \(\Span\{\U\}\).
It can be verified that
\begin{equation*}
\dist\bigl(\Span\{\X\},\Span\{\V_{1}\}\bigr)
=
\norm{(\I_{\n}-\X\X\trans)(\V_{1}-\U)}_{2}
\leq \sqrt{2}t
< \ubE,
\end{equation*}
and
\begin{equation*}
(\A\U-\U\Lam_1)
(\unitvec_{i}+\unitvec_{j})
=t(\lam_j-\lam_i)\z \neq \zeromat,\qquad
\skmat(\A\U-\U\Lam_1)
(\unitvec_{i}+\unitvec_{j})= \zeromat.
\end{equation*}
Thus, \(\Span\{\X,\A\X\}\) contains a nonzero vector
in \(\Null(\skmat)\), violating the lower embedding bound
for any \(\skeps \in [0, 1)\).
\end{proof}
Therefore, no fixed sketching matrix \(\skmat\) 
can satisfy this embedding assumption for all such joint subspaces
in any neighborhood of \(\Span\{\V_{1}\}\).
We therefore only assume that the fixed sketching matrix \(\skmat\)
is an \(\skeps\)-subspace embedding for \(\Span\{\V_{1}\}\).
Under this assumption, the following subsections establish
local equivalence of the classical and sketched eigensolvers.

\subsection{Local Equivalence of Classical and Sketched Eigensolvers}\label{subsec:local_equiv}

In this subsection, we show that, under the local embedding
assumption introduced in Section~\ref{subsec:local_behavior},
SI and sSI, as well as TraceMIN and sTraceMIN (when the correction equations are solved exactly),
generate identical search subspaces when \(\Span\{\X\iterind{0}\}\) is sufficiently close
to \(\Span\{\V_{1}\}\).
For SI and sSI, this follows directly from the fact
that the distance between the search subspace
and \(\Span\{\V_{1}\}\) is nonincreasing under inverse iteration.
This is formalized in the following theorem.

\begin{theorem}[Local equivalence of SI and sSI]\label{thm:local_equiv_SI_sSI}
Using the notation of~\eqref{eq:A_spectral_decomposition} 
and Lemma~\ref{lem:nearby_subspace_embedding},
suppose the fixed sketching matrix \(\skmat\)
satisfies~\eqref{eq:subspace_embedding} on \(\Span\{\V_{1}\}\) with \(0 < \skeps < 1\).
Let \(\X\iterind{k}\) and \(\XS\iterind{k}\) denote the bases of the search subspaces at iteration
\(k\) of exact SI and sSI, respectively.
Suppose \(\X\iterind{0}\), \(\XS\iterind{0}\), and \(\C\iterind{0}\)
satisfy~\eqref{eq:search_basis_relations}, where \(\C\iterind{0}\) is nonsingular.
For a fixed \(\ubE \in [0, \sklower / \skspnorm)\), suppose that
\begin{equation*}
\dist\bigl(\Span\{\X\iterind{0}\},\Span\{\V_{1}\}\bigr) \leq \ubE.
\end{equation*}
Then the whitening step in sSI is well-defined
at every iteration, and
\begin{equation*}
\Span\{\X\iterind{k}\} = \Span\{\XS\iterind{k}\}
\end{equation*}
for every \(k \geq 0\).
\end{theorem}

For TraceMIN and sTraceMIN, however, the bases \(\X\) and \(\XS\) are updated differently,
through the correction equations~\eqref{eq:TraceMIN_correction} and~\eqref{eq:sTraceMIN_correction}, respectively.
Nevertheless, the next Lemma~\ref{lem:one-step_equiv_exact_TraceMIN_sTraceMIN} shows that,
if both correction equations
for \(\X\) and \(\XS\) in~\eqref{eq:search_basis_relations}
are solvable and solved exactly,
then the two correction equations produce the same updated search subspace.

Before stating the lemma, we introduce some notation for the proof.
With the notations defined in
\eqref{eq:search_basis_relations} and~\eqref{eq:search_residuals}, we rewrite~\eqref{eq:sTraceMIN_correction} as
\begin{equation} \label{eq:sTraceMIN_correction_S}
\TrMINPsk \A\TrMINPsk \TrMINDelS
= \TrMINPsk \RS,
\qquad
\TrMINPsk \TrMINDelS = \TrMINDelS,
\end{equation}
where \(\TrMINPsk = \I_{\n} - \XS (\skmat\XS)\trans \skmat\).
Let
\begin{equation}\label{eq:TraceMIN_updated_bases}
\Y = \X - \TrMINDel,
\qquad
\YS = \XS - \TrMINDelS
\end{equation}
be the updated bases of TraceMIN and sTraceMIN, respectively.

\begin{lemma}[Equivalence of correction equations]\label{lem:one-step_equiv_exact_TraceMIN_sTraceMIN}
Let \(\X\) and \(\XS\) satisfy~\eqref{eq:search_basis_relations},
and suppose that
\((\skmat \XS)\trans \skmat \A^{-1} \XS\) is nonsingular.
Then both~\eqref{eq:TraceMIN_correction} and~\eqref{eq:sTraceMIN_correction_S}
have unique solutions, and
\begin{equation*}
\Span\{\YS\}
=
\Span\{\XS-\TrMINDelS\}
=
\Span\{\X-\TrMINDel\}
=
\Span\{\Y\}.
\end{equation*}
\end{lemma}

\begin{proof}
The existence and uniqueness of the constrained solution to~\eqref{eq:TraceMIN_correction}
follow from the positive definiteness of \(\TrMINP \A \TrMINP\) on \(\Null(\X\trans)\);
see~\cite[section~2.3]{Sameh_Wisniewski_1982}.
We therefore only prove the uniqueness and existence of the solution
to~\eqref{eq:sTraceMIN_correction_S}.
Suppose \(\v \in \bbR^{\n}\) satisfies \(\TrMINPsk \A \TrMINPsk \v = \zeromat\) and \(\TrMINPsk \v = \v\).
Since \(\A \TrMINPsk \v = \A \v \in \Null(\TrMINPsk)\) and \(\Null(\TrMINPsk) = \Span\{\XS\}\), \(\v\) can be represented by
\(\v = \A^{-1} \XS \c\) for some \(\c\).
Also note that \(\TrMINPsk \v = \v\) indicates \((\skmat\XS)\trans\skmat\v=\zeromat\).
Therefore, we have
\begin{equation*}
\zeromat
=
(\skmat\XS)\trans \skmat \v
=
(\skmat\XS)\trans \skmat\A^{-1} \XS \c.
\end{equation*}
By the assumption that \((\skmat\XS)\trans \skmat\A^{-1} \XS\) is nonsingular, we have \(\c = \zeromat\) and hence \(\v = \zeromat\).
Consequently, \(\TrMINPsk \A \TrMINPsk\) is nonsingular as an operator on \(\Span\{\TrMINPsk\}\),
and~\eqref{eq:sTraceMIN_correction_S} has a unique solution.

Using~\(\TrMINP \R = \TrMINP \A \X\) and~\(\TrMINP \TrMINDel = \TrMINDel\),~\eqref{eq:TraceMIN_correction} gives
\begin{equation*}
\TrMINP \A \Y
=
\TrMINP \A \X - \TrMINP \A \TrMINDel
=
\TrMINP \R - \TrMINP \A \TrMINP \TrMINDel
=
\zeromat.
\end{equation*}
A similar argument applied to~\eqref{eq:sTraceMIN_correction_S} gives \(\TrMINPsk\A\YS = \zeromat\).
Therefore,
\begin{equation*}
\Span\{\Y\}
\subseteq
\Span\{\A^{-1}\X\},
\qquad
\Span\{\YS\}
\subseteq
\Span\{\A^{-1}\XS\}
=
\Span\{\A^{-1}\X\},
\end{equation*}
where the last equality follows from~\eqref{eq:search_basis_relations}.
On the other hand, it can be shown from~\eqref{eq:TraceMIN_updated_bases} and~\eqref{eq:TraceMIN_correction} that
\begin{equation*}
\X\trans \Y
=
\X\trans \X - \X\trans \TrMINDel
=
\I_{\nex}.
\end{equation*}
Similarly, from~\eqref{eq:TraceMIN_updated_bases} and~\eqref{eq:sTraceMIN_correction_S},
we have \((\skmat \XS)\trans \skmat \YS= \I_{\nex}\).
Thus, both \(\Y\) and \(\YS\) are of rank \(\nex\).
Since \(\A\) is nonsingular and \(\X\) has rank \(\nex\),
\(\A^{-1}\Span\{\X\}\) has dimension \(\nex\),
which implies that
\(\Span\{\YS\}
=
\A^{-1}\Span\{\X\}
=\Span\{\Y\}\).
\end{proof}

Again, the assumption on the invertibility of \((\skmat \XS)\trans \skmat \A^{-1}\XS\)
can be guaranteed if \(\X\) is sufficiently close to the exact characteristic subspace \(\Span\{\V_{1}\}\),
which is summarized in the following theorem.

\begin{theorem}[Local unique solvability of the correction equation in sTraceMIN]
\label{thm:local_unique_solvability_correction_sTraceMIN}
Using the notation of Lemma~\ref{lem:nearby_subspace_embedding},
suppose that \(\ubE\geq0\) satisfies
\begin{equation}
\label{eq:quad_ineq_ubE_straceMIN}
\left(
\skspnorm^{2}
+
\frac{\lam_{1}}{\lam_{\nex}} \sklower^{2}
\right)
\ubE^{2}
+
2 \skupper \skspnorm\ubE
<
\frac{\lam_{1}}{\lam_{\nex}}
\sklower^{2}.
\end{equation}
Then \(\XS\trans\skmat\trans\skmat\A^{-1}\XS\) is nonsingular
for every orthonormal basis \(\X\) satisfying
\begin{equation*}
\dist\bigl(\Span\{\X\},\Span\{\V_{1}\}\bigr)\leq\ubE
\end{equation*}
and every \(\XS\) satisfying~\eqref{eq:search_basis_relations}.
Consequently, the correction equation~\eqref{eq:sTraceMIN_correction_S}
has a unique solution, and the updated bases
in~\eqref{eq:TraceMIN_updated_bases} satisfy
\(\Span\{\YS\}=\Span\{\Y\}\).
\end{theorem}

\begin{proof}

The assumed inequality implies
\(\ubE<1\) and \(\skspnorm\ubE<\sklower\),
so Lemma~\ref{lem:nearby_subspace_embedding} applies.
Set
\begin{equation*}
\projVX = \V_{1}\trans\X,
\qquad
\E
=
\X-\V_{1} \projVX
=
\bigl(\I_{\n} - \V_{1}\V_{1}\trans\bigr)\X
.
\end{equation*}
Then \(\norm{\E}_{2}\leq\ubE\), \(\norm{\projVX}_{2}\leq1\), and
\begin{equation*}
\projVX\trans\projVX
=
\I_{\nex}-\E\trans\E,
\qquad
\sv_{\min}(\projVX)^{2}\geq1-\ubE^{2}.
\end{equation*}
Define
\begin{equation*}
\M=(\skmat\X)\trans\skmat\A^{-1}\X,
\qquad
\M_{0}
=
\projVX\trans(\skmat\V_{1})\trans(\skmat\V_{1})
\Lam_{1}^{-1}\projVX.
\end{equation*}
The subspace embedding property gives
\begin{equation*}
\sv_{\min}(\M_{0})
\geq
\frac{\sklower^{2}\sv_{\min}(\projVX)^{2}}{\lam_{\nex}}
\geq
\frac{\sklower^{2}(1-\ubE^{2})}{\lam_{\nex}}.
\end{equation*}
Using \(\X=\V_{1}\projVX+\E\) and
\(\A^{-1}\V_{1}=\V_{1}\Lam_{1}^{-1}\), we obtain
\begin{equation*}
\M-\M_{0}
=
(\skmat\E)\trans(\skmat\V_{1})\Lam_{1}^{-1}\projVX
+
(\skmat\X)\trans\skmat\A^{-1}\E.
\end{equation*}
Since \(\norm{\skmat\V_{1}}_{2}\leq\skupper\),
\(\norm{\skmat\X}_{2}\leq\skupper+\skspnorm\ubE\), and
\(\norm{\A^{-1}}_{2}=1/\lam_{1}\), it follows from our assumption that
\begin{equation*}
\norm{\M-\M_{0}}_{2}
\leq
\frac{\skspnorm^{2}\ubE^{2} + 2\skupper\skspnorm\ubE}{\lam_{1}}
<
\frac{\sklower^{2}(1-\ubE^{2})}{\lam_{\nex}}
\leq
\sv_{\min}(\M_{0}).
\end{equation*}
Thus, \(\M\) is nonsingular.
By~\eqref{eq:search_basis_relations}, \(\C\) is nonsingular and
\begin{equation*}
\XS\trans\skmat\trans\skmat\A^{-1}\XS
=
\C\trans\M\C.
\end{equation*}
The remaining conclusions follow from
Lemma~\ref{lem:one-step_equiv_exact_TraceMIN_sTraceMIN}.
\end{proof}

Combining Lemma~\ref{lem:one-step_equiv_exact_TraceMIN_sTraceMIN},
Theorem~\ref{thm:local_unique_solvability_correction_sTraceMIN},
and the nonincreasing subspace distance under exact inverse iteration,
we obtain the following result:

\begin{theorem}[Local equivalence of TraceMIN and sTraceMIN]\label{thm:local_equiv_TraceMIN_sTraceMIN}
Using the notation of~\eqref{eq:A_spectral_decomposition} 
and Lemma~\ref{lem:nearby_subspace_embedding},
suppose the fixed sketching matrix \(\skmat\) satisfies~\eqref{eq:subspace_embedding}
on \(\Span\{\V_{1}\}\) with \(0 < \skeps < 1\).
Let \(\X\iterind{k}\) and \(\XS\iterind{k}\) denote the bases of the search subspaces
at iteration \(k\) of TraceMIN and sTraceMIN, respectively.
Assume that the correction equations~\eqref{eq:TraceMIN_correction} and~\eqref{eq:sTraceMIN_correction_S} are solved exactly.
Suppose \(\X\iterind{0}\), \(\XS\iterind{0}\), and \(\C\iterind{0}\)
satisfy~\eqref{eq:search_basis_relations}, where \(\C\iterind{0}\) is nonsingular.
Let \(\ubE\geq0\) satisfy the quadratic inequality~\eqref{eq:quad_ineq_ubE_straceMIN} in
Theorem~\ref{thm:local_unique_solvability_correction_sTraceMIN},
and assume that
\begin{equation*}
\dist\bigl(\Span\{\X\iterind{0}\},\Span\{\V_{1}\}\bigr)
\leq\ubE.
\end{equation*}
Then, at every iteration, the correction equations
\eqref{eq:TraceMIN_correction} and~\eqref{eq:sTraceMIN_correction_S}
have unique constrained solutions, the basis
whitening step is well-defined, and
\begin{equation*}
\Span\{\X\iterind{k}\}=\Span\{\XS\iterind{k}\}
\end{equation*}
for every \(k \geq 0\).
\end{theorem}

\subsection{Local Closeness of the sRR and Target Subspaces} \label{subsec:local_closeness_RR_sRR}

As shown in Section~\ref{subsec:local_equiv},
when near convergence, SI and sSI, as well as TraceMIN and sTraceMIN,
generate exactly the same search subspaces.
However, in sketched algorithms, the subspace selection step is performed
on the sketched matrix \(\Apsk\), and we use its real Schur decomposition
for subspace selection and convergence tests.
To give a full picture of the two sketched eigensolvers, we need to compare
the \(\nev\)-dimensional subspace obtained from the real Schur decomposition
of \(\Apsk\) with the characteristic subspace associated with the \(\nev\) smallest
eigenvalues of \(\A\).

We further split the matrices in~\eqref{eq:A_spectral_decomposition} as
\begin{equation*}
\V_{1}=[\V_{1;1},\V_{1;2}],
\qquad
\Lam_{1}=\diag\{\Lam_{1;1},\Lam_{1;2}\},
\end{equation*}
where \(\V_{1;1}\in\bbR^{\n\times\nev}\) and
\(\Lam_{1;1}=\diag\{\lam_1,\ldots,\lam_{\nev}\}\).
We first state a specialization of~\cite[Theorem~3.1]{Karow_Kressner_2014}
and derive a bound on the corresponding subspace distance.

\begin{theorem}[Subspace gap under a nonsymmetric perturbation]
\label{thm:karow_kressner_special_case}
Using the notation of~\eqref{eq:A_spectral_decomposition}, let
\(\V=[\V_{1},\V_{2}]\) and
\(\evgap=\lam_{\nev+1}-\lam_{\nev}>0\).
Let \(\Perturb\in\bbR^{\n\times\n}\) be a real perturbation, and set
\(\Perturbnorm=\norm{\Perturb}_{2}\) and
\(\PerturbVnorm=\norm{\Perturb\V_{1;1}}_{2}\).
Suppose that \(\Perturbnorm<\evgap/2\).
Then \(\A+\Perturb\) has exactly \(\nev\) eigenvalues,
counted with algebraic multiplicity, in the half-plane
\begin{equation*}
\left\{
z\in\bbC:
\Re z<\frac{\lam_{\nev}+\lam_{\nev+1}}{2}
\right\}.
\end{equation*}
The associated real spectral characteristic subspace has the form
\begin{equation*}
\Span\left\{
\V\begin{bmatrix}\I_{\nev}\\\N\end{bmatrix}
\right\},
\qquad
\N\in\bbR^{(\n-\nev)\times\nev},
\qquad
\norm{\N}_2\leq\frac{2\Perturbnorm}{\evgap}<1.
\end{equation*}
In particular,
\begin{equation}\label{eq:nonsymmetric_perturbation_gap}
\dist\left(
\Span\{\V_{1;1}\},
\Span\left\{\V\begin{bmatrix}\I_{\nev}\\\N\end{bmatrix}\right\}
\right)
\leq
\min\biggl\{
\frac{2 \Perturbnorm}{\evgap},
\frac{\PerturbVnorm}{\evgap-2\Perturbnorm}
\biggr\}.
\end{equation}
\end{theorem}

\begin{proof}
The eigenvalue count and characteristic subspace representation follow from the specialization
of~\cite[Theorem~3.1]{Karow_Kressner_2014}.
Since \(\V\) is orthogonal,
\begin{equation}\label{eq:perturbed_subspace_gap_graph}
\begin{aligned}
\mygap
&=\dist\left(
\Span\{\V_{1;1}\},
\Span\left\{\V\begin{bmatrix}\I_{\nev}\\\N\end{bmatrix}\right\}
\right)\\
&=\norm{\N(\I_{\nev}+\N\trans\N)^{-1/2}}_2
=\frac{\norm{\N}_2}{\sqrt{1+\norm{\N}_2^2}}
\geq\frac{\norm{\N}_2}{2},
\end{aligned}
\end{equation}
where the last inequality uses \(\norm{\N}_2<1\).
The equation~\eqref{eq:perturbed_subspace_gap_graph} gives the first bound in~\eqref{eq:nonsymmetric_perturbation_gap}.
On the other hand, the invariance of this characteristic subspace gives
\begin{equation*}
[-\N,\I_{\n-\nev}]\V\trans(\A+\Perturb)\V
\begin{bmatrix}\I_{\nev}\\\N\end{bmatrix}
=\zeromat,
\end{equation*}
or equivalently,
\begin{equation*}
\diag\{\Lam_{1;2},\Lam_2\}\N-\N\Lam_{1;1}
=
-[-\N,\I_{\n-\nev}]\V\trans\Perturb
\bigl(\V_{1;1}+[\V_{1;2},\V_{2}]\N\bigr).
\end{equation*}
Since the spectra of \(\Lam_{1;1}\) and
\(\diag\{\Lam_{1;2},\Lam_2\}\) are separated by \(\evgap\), we obtain
\begin{equation*}
\begin{aligned}
\evgap\norm{\N}_2
&\leq\norm{\diag\{\Lam_{1;2},\Lam_2\}\N-\N\Lam_{1;1}}_2\\
&\leq\sqrt{1+\norm{\N}_2^2}
\norm{\Perturb\bigl(\V_{1;1}+[\V_{1;2},\V_{2}]\N\bigr)}_2\\
&\leq\sqrt{1+\norm{\N}_2^2}
\bigl(\PerturbVnorm+\Perturbnorm\norm{\N}_2\bigr).
\end{aligned}
\end{equation*}
Dividing by \(\sqrt{1+\norm{\N}_2^2}\)
and using~\eqref{eq:perturbed_subspace_gap_graph} gives
\begin{equation*}
\evgap\mygap
\leq\PerturbVnorm+\Perturbnorm\norm{\N}_2
\leq\PerturbVnorm+2\Perturbnorm\mygap.
\end{equation*}
Since \(\evgap-2\Perturbnorm>0\), this proves the second bound in~\eqref{eq:nonsymmetric_perturbation_gap}.
\end{proof}

We now consider the subspace selected by sRR.
The real Schur decomposition of \(\Apsk\) is
\begin{equation}\label{eq:sRR_schur_decomposition}
\Apsk
=
\ZS \T \ZS\trans,
\qquad
\ZS\trans \ZS = \I_{\nex},
\end{equation}
with the diagonal blocks of \(\T\) ordered by nondecreasing real parts of
their eigenvalues.
When \(\T(\nev+1,\nev) = 0\), we can partition \(\ZS\) and \(\T\) as
\begin{equation}\label{eq:sRR_schur_partition}
\ZS = [\Z_{\skmat, 1}, \Z_{\skmat, 2}],
\qquad
\T=
\begin{bmatrix}
\T_{1, 1}& \T_{1, 2}\\
&\T_{2, 2}
\end{bmatrix},
\qquad
\Z_{\skmat, 1} \in \bbR^{\nex \times \nev}.
\end{equation}
The following theorem gives sufficient conditions for this partition to be
well-defined and bounds the gap between the target characteristic subspace
\(\Span\{\V_{1;1}\}\) and the subspace
\(\Span\{\XS\Z_{\skmat,1}\}\) selected by sRR.
As in Theorems~\ref{thm:local_equiv_SI_sSI}
and~\ref{thm:local_equiv_TraceMIN_sTraceMIN},
\(\skmat\) is only required to embed \(\Span\{\V_{1}\}\);
no embedding assumption is imposed on \(\Span\{\X,\A\X\}\).

\begin{theorem}[Local gap between the target and real Schur sRR subspaces]
\label{thm:local_gap_RR_sRR}
Using the notation of~\eqref{eq:A_spectral_decomposition}
and Lemma~\ref{lem:nearby_subspace_embedding},
suppose the fixed sketching matrix \(\skmat\)
satisfies~\eqref{eq:subspace_embedding}
on \(\Span\{\V_{1}\}\) with \(0<\skeps<1\).
Let \(\X\), \(\XS\), and \(\C\)
satisfy~\eqref{eq:search_basis_relations},
and consider the real Schur decomposition
in~\eqref{eq:sRR_schur_decomposition}.
Assume that
\(\evgap=\lam_{\nev+1}-\lam_{\nev}>0\),
and set \(\omega=\lam_{\n}-\lam_1\).
Fix
\begin{equation*}
0\leq\ubE
<
\frac{\sklower\evgap}
{\skspnorm(\evgap+2\omega)}
< 1,
\end{equation*}
and assume that
\begin{equation*}
\dist\bigl(\Span\{\X\},\Span\{\V_{1}\}\bigr)
\leq\ubE.
\end{equation*}
Then the selection of the first \(\nev\) Schur vectors
in~\eqref{eq:sRR_schur_decomposition}
does not split any \(2\times2\) diagonal block,
so the partition in~\eqref{eq:sRR_schur_partition}
is well-defined.
Moreover,
\begin{equation*}
\dist\bigl(
\Span\{\V_{1;1}\},
\Span\{\XS\Z_{\skmat,1}\}
\bigr)
\leq
\min\biggl\{
\frac{2\skspnorm \omega \ubE}{\evgap (\sklower - \skspnorm \ubE)},
\frac{\skspnorm\omega\ubE}
{\sklower\evgap-\skspnorm(\evgap+2\omega)\ubE}
\biggr\}.
\end{equation*}
\end{theorem}

\begin{proof}

Define a perturbation of \(\A\) by
\begin{equation*}
\Perturb=-\TrMINPsk\R\X\trans,
\qquad
\Perturbnorm=\norm{\Perturb}_{2},
\end{equation*}
where \(\TrMINPsk=\I_{\n}-\XS\XS\trans\skmat\trans\skmat\) is the sketched projector and \(\R=\A\X-\X\Ap\).
The relationship between the spectral of \(\A\) and \(\Apsk\) can be represented by
\begin{equation}\label{eq:original_matrix_perturbation}
(\A+\Perturb)\X
=\X(\C\Apsk\C^{-1}).
\end{equation}
Therefore, what we need is to show that the subspace selected
by sRR corresponds to the \(\nev\) eigenvalues of
\(\A+\Perturb\) with the smallest real parts, and then apply
Theorem~\ref{thm:karow_kressner_special_case}
to bound its distance from \(\Span\{\V_{1;1}\}\).

We first estimate \(\norm{\Perturb}_{2}\).
On the one hand, using Lemma~\ref{lem:nearby_subspace_embedding},
we have
\begin{equation*}
\norm{\TrMINPsk}_{2}
=
\norm{\XS\XS\trans\skmat\trans\skmat}_{2}
\leq \norm{\XS}_{2} \norm{\skmat \XS}_{2} \norm{\skmat}_{2}
\leq \frac{\skspnorm}{\sklower-\skspnorm\ubE}.
\end{equation*}
On the other hand, by define \(\TrMINP=\I_{\n}-\X\X\trans\) and \(\mu=(\lam_1+\lam_{\n})/2\), we obtain
\(\R=\A\X-\X\Ap=\TrMINP\A\X=\TrMINP(\A-\mu\I_{\n})\X\).
Similar to the proof of Theorem~\ref{thm:local_unique_solvability_correction_sTraceMIN}, by setting \(\projVX=\V_{1}\trans\X\) and 
\(\E=(\I_{\n}-\V_{1}\V_{1}\trans)\X\),
we can represent \(\X\) in
\begin{equation*}
\X=\V_{1}\projVX+\E,\qquad
\norm{\projVX}_{2}\leq1,
\qquad
\norm{\E}_{2}\leq\ubE.
\end{equation*}
Then, the residual satisfies
\begin{equation*}
\begin{aligned}
\norm{\R}_2
&=\norm{\TrMINP(\A-\mu\I_{\n})\X}_2
=\norm{\TrMINP\V_{1}(\Lam_1-\mu\I_{\nex})\projVX
+\TrMINP(\A-\mu\I_{\n})\E}_2\\
&\leq
\frac{\omega}{2}
\bigl(
\norm{\TrMINP\V_{1}}_{2}\norm{\projVX}_{2}
+\norm{\E}_{2}
\bigr)
\leq\omega\ubE,
\end{aligned}
\end{equation*}
where the last inequality uses \(\norm{\TrMINP\V_{1}}_{2} = \dist\bigl(\Span\{\X\},\Span\{\V_{1}\}\bigr) \leq\ubE\).
The preceding estimates and our assumption imply
\begin{equation} \label{eq:norm_perturb}
\norm{\Perturb}
\leq \norm{\TrMINPsk}_2\norm{\R}_{2}
\leq 
\frac{\skspnorm\omega\ubE}{\sklower-\skspnorm\ubE}
<\frac{\evgap}{2}.
\end{equation}

We next verify that the \(\nev\) eigenvalues of \(\A+\Perturb\)
identified in Theorem~\ref{thm:karow_kressner_special_case}
are eigenvalues of \(\Apsk\).
Choose an orthogonal completion \([\X,\X_{\perp}]\).
Since \(\Perturb\X_{\perp}=\zeromat\),
\eqref{eq:original_matrix_perturbation} gives
\begin{equation}\label{eq:perturbed_matrix_block_form}
[\X,\X_{\perp}]\trans(\A+\Perturb)[\X,\X_{\perp}]
=
\begin{bmatrix}
\C\Apsk\C^{-1} & \X\trans\A\X_{\perp}\\
\zeromat & \X_{\perp}\trans\A\X_{\perp}
\end{bmatrix}.
\end{equation}
The subspace distance assumption and the bound on \(\ubE\) give
\begin{equation*}
\lambda_{\min}(\X_{\perp}\trans\A\X_{\perp})
\geq \lam_{\nex+1}-\omega\ubE^2
> \frac{\lam_{\nev}+\lam_{\nev+1}}{2}.
\end{equation*}
To see this, take a unit vector \(\z\in\Span\{\X_{\perp}\}\),
the assumption gives \(\norm{\V_{1} \trans \z}_{2}\leq\ubE\).
Using the spectral decomposition of \(\A\), we obtain
\begin{equation*}
\z\trans\A \z
\geq\lam_{1}\norm{\V_{1}\V_{1}\trans \z}_{2}^{2}
+\lam_{\nex+1}\norm{\z-\V_{1}\V_{1}\trans \z}_{2}^{2}
\geq\lam_{\nex+1}-\omega\ubE^{2}.
\end{equation*}
And the second inequality follows from
\begin{equation*}
\omega \ubE^{2}
\leq
\omega \ubE
\leq
\frac{\skspnorm\omega \ubE}{\sklower}
\leq
\frac{\skspnorm\omega\ubE}{\sklower-\skspnorm\ubE}
<
\frac{\evgap}{2}
=
\frac{\lam_{\nev+1}-\lam_{\nev}}{2}.
\end{equation*}
Thus, the complementary block contributes no eigenvalues
in the left half-plane of
Theorem~\ref{thm:karow_kressner_special_case}.
By~\eqref{eq:perturbed_matrix_block_form} and similarity,
\(\Apsk\) has exactly \(\nev\) eigenvalues there.
Since conjugate eigenvalues have the same real part,
selecting the first \(\nev\) Schur vectors does not split
any \(2\times2\) diagonal block.

Furthermore, \eqref{eq:original_matrix_perturbation} gives
\begin{equation*}
(\A+\Perturb)\XS\Z_{\skmat,1}
=
\XS\Z_{\skmat,1}\T_{1,1}.
\end{equation*}
Hence, \(\Span\{\XS\Z_{\skmat,1}\}\) is the real spectral
characteristic subspace of \(\A+\Perturb\) associated with these
\(\nev\) eigenvalues.
Since \(\norm{\Perturb\V_{1;1}}_{2}\leq\Perturbnorm<\evgap/2\),
Theorem~\ref{thm:karow_kressner_special_case}
and~\eqref{eq:norm_perturb} yield
\begin{equation*}
\begin{aligned}
\dist\bigl(
\Span\{\V_{1;1}\},\Span\{\XS\Z_{\skmat,1}\}
\bigr)
& \leq
\min\biggl\{
\frac{2\Perturbnorm}{\evgap},
\frac{\Perturbnorm}{\evgap-2\Perturbnorm}
\biggr\} \\
& \leq
\min\biggl\{
\frac{2 \skspnorm \omega \ubE}{\evgap (\sklower - \skspnorm \ubE)},
\frac{\skspnorm\omega\ubE}
{\sklower\evgap-\skspnorm(\evgap+2\omega)\ubE}
\biggr\}.
\end{aligned}
\end{equation*}
This completes the proof.
Note that the first bound is less than \(1\) from~\eqref{eq:norm_perturb}.
\end{proof}

Theorem~\ref{thm:local_gap_RR_sRR}
also establishes local convergence of the selected subspaces.
Under the assumptions of
Theorem~\ref{thm:local_equiv_SI_sSI} or
Theorem~\ref{thm:local_equiv_TraceMIN_sTraceMIN},
the common \(\nex\)-dimensional search subspaces converge
to \(\Span\{\V_{1}\}\).
Define
\begin{equation*}
\ubE_k
=
\dist\bigl(
\Span\{\X\iterind{k}\},
\Span\{\V_{1}\}
\bigr)
\rightarrow0.
\end{equation*}
For sufficiently large \(k\), the radius condition in
Theorem~\ref{thm:local_gap_RR_sRR}
holds with \(\ubE=\ubE_k\),
and the partition in~\eqref{eq:sRR_schur_partition}
is well-defined.
The corresponding subspace distance bound is
\(\bigO(\ubE_k)\), so the selected sRR subspaces
converge to the target characteristic subspace:
\begin{equation*}
\dist\bigl(
\Span\{\XS\iterind{k}\Z_{\skmat,1}\iterind{k}\},
\Span\{\V_{1;1}\}
\bigr)
\rightarrow0.
\end{equation*}

\section{Numerical Experiments} \label{sec:experiments}

\subsection{Experiment Settings} \label{subsec:experiment_settings}

\begin{table}[tbhp]
\centering
\caption{Test matrices before shifting.
The scalars \(\n\) and \(\nnz(\A_{0})\) are the matrix size and number of nonzero entries, respectively.}
\begin{tabular}{rlrr}
\toprule
No. & Matrix \(\A_{0}\) & \(\n\) & \(\nnz(\A_{0})\) \\
\midrule
1  & \texttt{bcsstm21}      &   \(3\,600\) &       \(3\,600\) \\
2  & \texttt{rail\_5177}    &   \(5\,177\) &      \(35\,185\) \\
3  & \texttt{Muu}          &   \(7\,102\) &     \(170\,134\) \\
4  & \texttt{fv1}          &   \(9\,604\) &      \(85\,264\) \\
5  & \texttt{shuttle\_eddy} &  \(10\,429\) &     \(103\,599\) \\
6  & \texttt{barth5}       &  \(15\,606\) &     \(107\,362\) \\
7  & \texttt{Si5H12}       &  \(19\,896\) &     \(738\,598\) \\
8  & \texttt{mario001}     &  \(38\,434\) &     \(204\,912\) \\
9  & \texttt{c-65}         &  \(48\,066\) &     \(360\,428\) \\
10 & \texttt{Andrews}      &  \(60\,000\) &     \(760\,154\) \\
11 & \texttt{Ga3As3H12}    &  \(61\,349\) & \(5\,970\,947\) \\
12 & \texttt{Ga10As10H30}  & \(113\,081\) & \(6\,115\,633\) \\
\bottomrule
\end{tabular}
\label{tab:test_matrices}
\end{table}

All experiments are implemented in double precision using MATLAB R2026a 
on a server equipped with two Intel Xeon Gold 6226R processors 
with a base frequency of 2.90 GHz and 1 TiB of RAM.
The code for reproducing the numerical experiments is available
at \url{https://github.com/JingyuLiuMath/sSI_and_sTraceMIN}.
Our test set consists of the twelve real symmetric matrices from the
SuiteSparse Matrix Collection used
in~\cite{Liu_Yuqi_Ma_Yuxin_Shao_Meiyue_2026}.
Table~\ref{tab:test_matrices} lists their sizes and numbers of nonzero
entries.
To make the matrices in Table~\ref{tab:test_matrices} positive definite,
we shift them by
\begin{equation*}
\A=\A_{0}+\sigma\I,
\qquad
\sigma=\max\bigl\{0,-1.05\cdot\widehat{\lambda}_{\min},
10^{-4}\norm{\A_{0}}_{\infty}-\widehat{\lambda}_{\min}\bigr\},
\end{equation*}
where \(\widehat{\lambda}_{\min}\) is estimated by MATLAB's
\texttt{eigs} with the option \texttt{smallestreal}.
All our numerical experiments are performed on the shifted matrices \(\A\).
For each \(\A\), we compute its \(\nev = 100\) smallest
eigenvalues and their corresponding eigenvectors.
The dimension of the search subspace is set to
\(\nex=\lceil 1.5 \nev\rceil = 150\).

The stopping criteria are
\(\resnorm\leq\tol\) for the standard methods 
and \(\skresnorm\leq\tol\) for the sketched methods,
where \(\resnorm\) and \(\skresnorm\)
are the relative block residuals defined
in~\eqref{eq:rel_res_entry}
and~\eqref{eq:rel_sres_entry}, respectively.
For all algorithms, we set \(\tol=10^{-10}\)
and limit the number of iterations to \(1000\).

The sketching matrices \(\skmat\) are sparse sign sketches with eight nonzero
entries per column~\cite[Section~9.2]{Martinsson_Tropp_2020},
for both sSI and sTraceMIN.
In each column, the row indices are sampled uniformly without
replacement, and the nonzero entries independently take the values
\(\pm 1/ \sqrt{8}\) with equal probability.
We compare sketch dimensions \(\skdim = \nu \nex\) for
\(\nu\in\{2,4,8\}\), corresponding to \(\skdim = 300, 600, 1200\).

To better reflect the performance of practical implementations,
we employ several acceleration techniques.
Both SI and sSI factorize \(\A\) once using MATLAB's
\texttt{decomposition} and reuse the factorization for direct solves in
the inverse iteration steps.
Both TraceMIN and sTraceMIN solve their correction equations
inexactly by performing \(5\) GMRES iterations 
for each right-hand side~\cite{Saad_2003},
with a zero initial guess and no additional preconditioner.
Soft locking is employed for all methods~\cite{Duersch_Shao_Yang_Gu_2018},
with a locking threshold of \(\tol / 100=10^{-12}\).
Locked vectors remain in the search subspace and participate in the RR
or sRR procedure, but are excluded from the inverse or correction solves.
For all four methods, the final relative residuals reported in the
tables are recomputed in the original space
using~\eqref{eq:rel_res_entry}, independently of the outer stopping test.

We perform each configuration three times with different random seeds,
and select the ones with the median total runtime
to report its iteration count, final relative residual
and timing components in the tables and figures.
Within each trial, all four methods use the same initial Gaussian block.
For each sketch dimension and seed, sSI and sTraceMIN use the same
sketching matrix, which remains fixed throughout each run.
The total solver time includes 
estimation of \(\norm{\A}_2\),
matrix factorization for SI and sSI,
outer iterations, all linear solves, final postprocessing
and residual verification.
It excludes matrix loading, formation of the shifted matrix, 
and generation of the initial Gaussian block and the sketching matrix \(\skmat\).

\subsection{Experiments on SI and sSI} \label{subsec:experiment_SI}

Table~\ref{tab:SI} compares the performance of SI and sSI,
with three sketch dimensions considered for sSI.
Since sSI uses \(\skresnorm\) to assess convergence, its final
relative residual \(\resnorm\) in the original space may slightly
exceed the prescribed tolerance \(\tol=10^{-10}\).
Nevertheless, the final relative residuals reported in
Table~\ref{tab:SI} do not exceed \(1.3\times10^{-10}\).
Thus, both SI and sSI achieve residuals comparable to \(\tol\) on all
twelve test problems.

\begin{table}[tb]
\centering
\caption{Runtime in seconds, outer iteration count and final relative
residual in the original space for SI and sSI.
Problem numbers follow Table~\ref{tab:test_matrices}.
We use \(\nev=100\), \(\nex=150\) and \(\skdim=\nu\nex\).
Each configuration is represented by the trial with median total runtime.
The column ``save'' gives the percentage of time saved relative to SI.}
\resizebox{\textwidth}{!}{%
\begin{tabular}{lrrrrrrrrrrrrrrr}
\toprule
No. & \multicolumn{3}{c}{SI} & \multicolumn{4}{c}{sSI ($\nu$ = 2)} & \multicolumn{4}{c}{sSI ($\nu$ = 4)} & \multicolumn{4}{c}{sSI ($\nu$ = 8)} \\
\cmidrule(lr){2-4}\cmidrule(lr){5-8}\cmidrule(lr){9-12}\cmidrule(lr){13-16}
 & time & iter & relres & time & iter & relres & save & time & iter & relres & save & time & iter & relres & save \\
\midrule
1 & 6.5e-01 & 30 & 7.9e-11 & 7.3e-01 & 30 & 7.6e-11 & -12\% & 8.1e-01 & 30 & 7.5e-11 & -24\% & 9.7e-01 & 30 & 7.6e-11 & -48\% \\
2 & 6.8e+00 & 192 & 9.6e-11 & 5.3e+00 & 191 & 1.3e-10 & 22\% & 5.8e+00 & 187 & 1.0e-10 & 14\% & 6.4e+00 & 192 & 9.9e-11 & 5\% \\
3 & 6.5e+00 & 112 & 9.2e-11 & 5.0e+00 & 109 & 1.2e-10 & 23\% & 5.3e+00 & 109 & 1.1e-10 & 19\% & 5.6e+00 & 109 & 1.1e-10 & 14\% \\
4 & 1.5e+01 & 187 & 9.2e-11 & 1.1e+01 & 192 & 1.3e-10 & 23\% & 1.2e+01 & 193 & 1.1e-10 & 18\% & 1.3e+01 & 186 & 1.0e-10 & 13\% \\
5 & 1.5e+01 & 223 & 1.0e-10 & 1.1e+01 & 221 & 1.3e-10 & 28\% & 1.2e+01 & 223 & 1.0e-10 & 24\% & 1.3e+01 & 222 & 1.1e-10 & 17\% \\
6 & 2.7e+01 & 227 & 1.0e-10 & 1.9e+01 & 229 & 1.2e-10 & 29\% & 2.0e+01 & 226 & 1.1e-10 & 28\% & 2.1e+01 & 229 & 1.0e-10 & 23\% \\
7 & 6.9e+01 & 108 & 9.3e-11 & 6.2e+01 & 105 & 1.2e-10 & 11\% & 6.2e+01 & 108 & 9.8e-11 & 10\% & 6.4e+01 & 108 & 9.5e-11 & 7\% \\
8 & 8.6e+01 & 308 & 9.8e-11 & 5.3e+01 & 306 & 1.3e-10 & 39\% & 5.5e+01 & 308 & 1.1e-10 & 36\% & 5.8e+01 & 308 & 1.0e-10 & 33\% \\
9 & 3.3e+02 & 875 & 1.0e-10 & 2.0e+02 & 866 & 1.3e-10 & 39\% & 2.1e+02 & 846 & 1.1e-10 & 37\% & 2.2e+02 & 843 & 1.0e-10 & 34\% \\
10 & 1.4e+02 & 81 & 8.8e-11 & 1.3e+02 & 80 & 1.3e-10 & 11\% & 1.3e+02 & 83 & 9.6e-11 & 12\% & 1.2e+02 & 82 & 9.3e-11 & 14\% \\
11 & 2.9e+02 & 91 & 9.0e-11 & 2.5e+02 & 90 & 1.2e-10 & 16\% & 2.5e+02 & 91 & 1.1e-10 & 14\% & 2.6e+02 & 91 & 9.2e-11 & 12\% \\
12 & 6.8e+02 & 118 & 9.1e-11 & 6.0e+02 & 117 & 1.2e-10 & 12\% & 6.0e+02 & 117 & 1.1e-10 & 11\% & 6.2e+02 & 118 & 9.3e-11 & 8\% \\
\bottomrule
\end{tabular}
}
\label{tab:SI}
\end{table}

The iteration counts of sSI remain close to those of SI for
all three choices of \(\nu\), with a maximum relative increase
of approximately \(3.2\%\).
This observation is further supported by the residual histories
for \texttt{mario001} and \texttt{c-65}
in Figure~\ref{fig:SI_residuals}.
For each problem, the SI curve and the three sSI curves with
\(\nu=2,4,8\) exhibit nearly identical residual decay per iteration.
Similar convergence behavior is also observed for the other ten test problems.
These results indicate that sSI achieves nearly the same
convergence rate as SI on these test problems.

\begin{figure}[tbhp]
\centering
\includegraphics[width=0.48\textwidth]{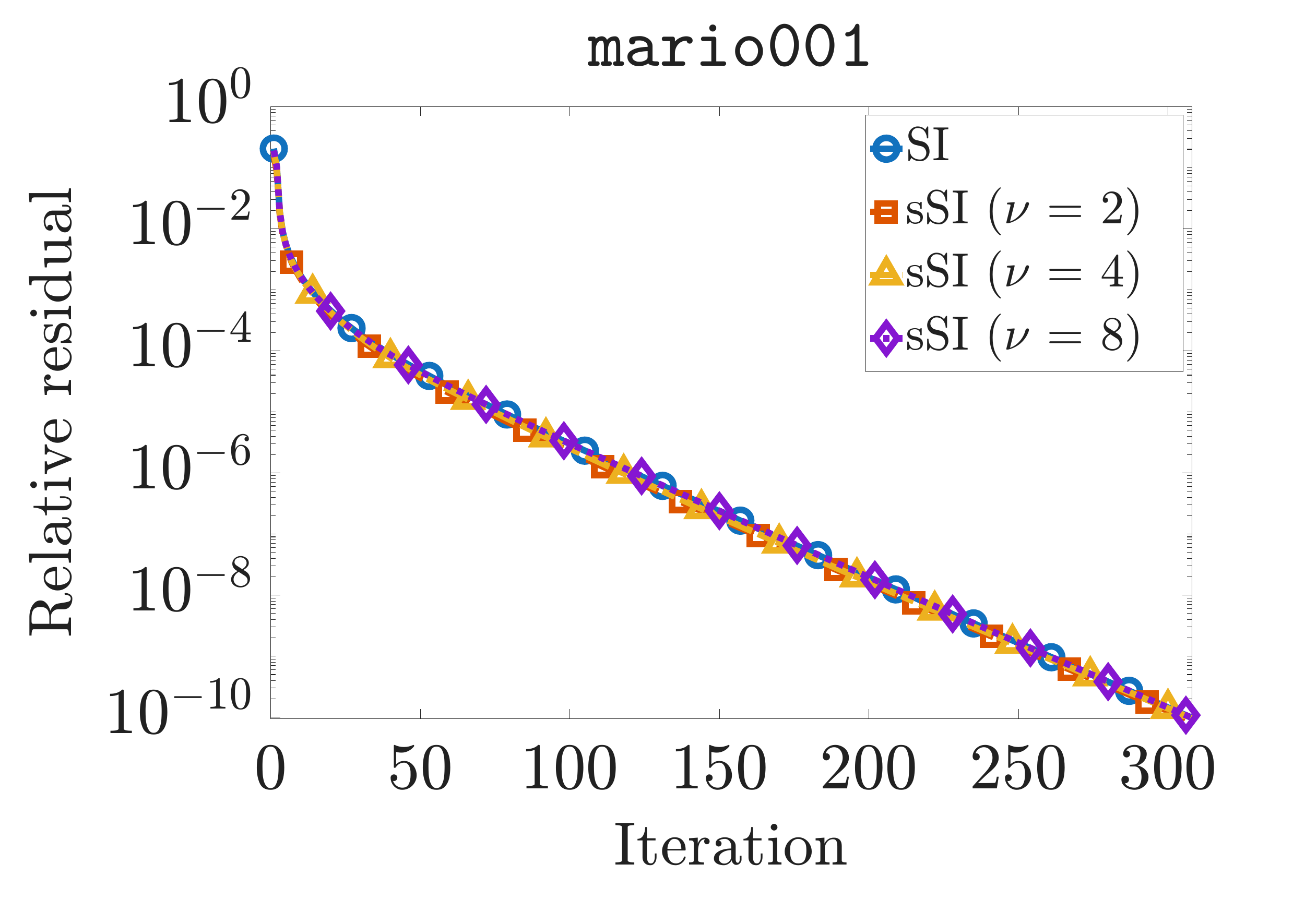}
\hfill
\includegraphics[width=0.48\textwidth]{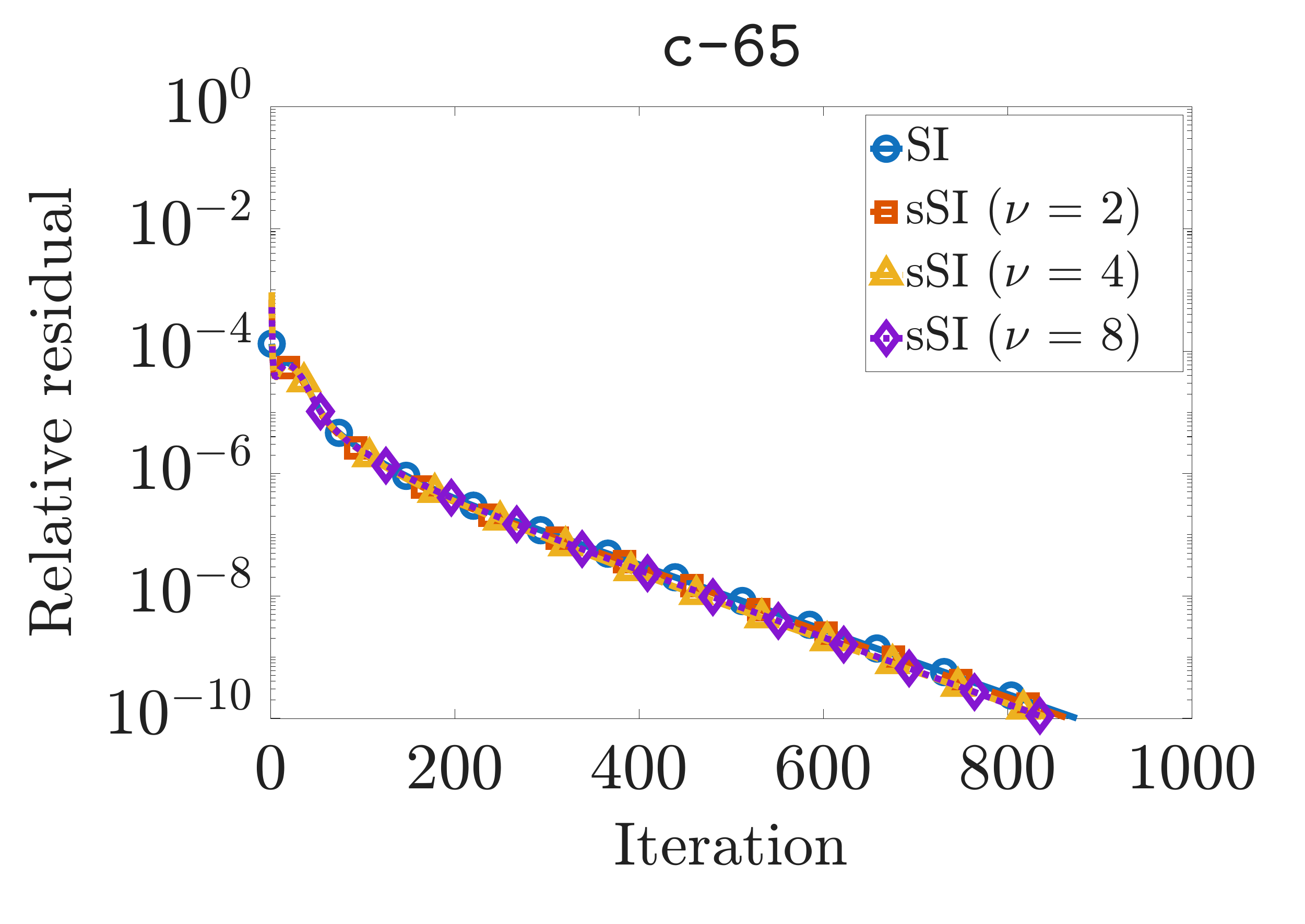}
\caption{Convergence histories of SI and sSI for \texttt{mario001}~(left, No.~8) 
and \texttt{c-65}~(right, No.~9).
Both methods compute \(\nev=100\) smallest eigenpairs using
\(\nex=150\) search vectors.
For sSI, \(\skdim = \nu \nex\) is the sketch dimension.
SI plots the relative block residual \(\resnorm\)
in~\eqref{eq:rel_res_entry}, 
whereas sSI plots its sketched counterpart
\(\skresnorm\) in~\eqref{eq:rel_sres_entry}.
Each curve shows the run with median total runtime among three trials.}
\label{fig:SI_residuals}
\end{figure}

On the other hand, the use of sketching can reduce the computational cost per iteration.
Figure~\ref{fig:SI_ratio_cluster} shows substantial reductions
in the time spent on basis orthonormalization 
(``Orth / Whitening'')
and residual evaluation (``Residual'').
These savings outweigh the additional cost of applying and updating the sketches (``Basis update''),
resulting in an overall speedup.

\begin{figure}[tb]
\centering
\includegraphics[width=0.98\linewidth]{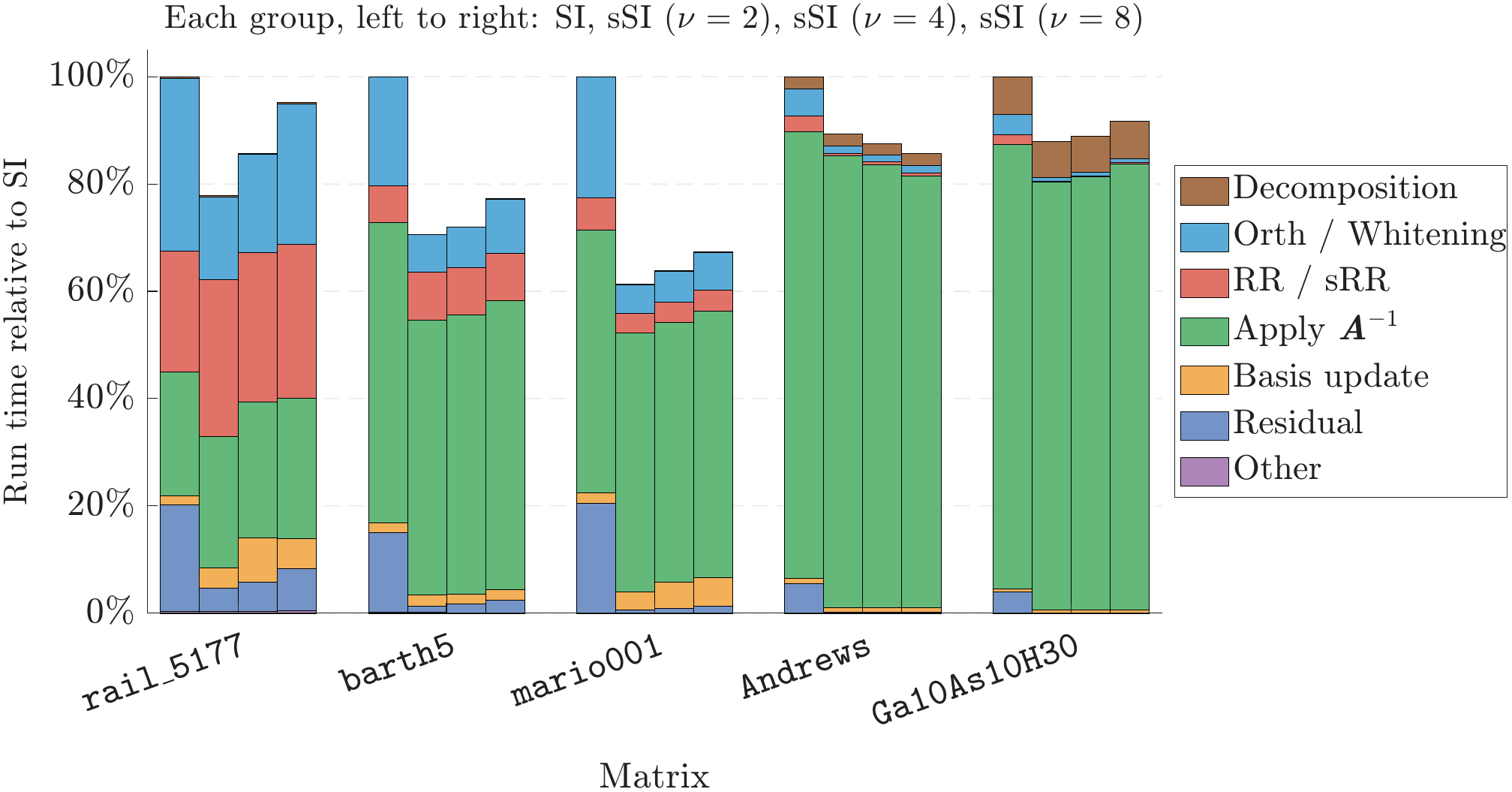}
\caption{Runtime breakdown for \texttt{rail\_5177},
\texttt{barth5}, \texttt{mario001},
\texttt{Andrews} and \texttt{Ga10As10H30}
in Table~\ref{tab:test_matrices}.
Within each group, the bars show SI and sSI with \(\nu=2,4,8\),
from left to right.
Both methods compute \(\nev=100\) smallest eigenpairs using
\(\nex=150\) search vectors.
For sSI, \(\skdim = \nu \nex\) is the sketch dimension.
Each bar represents the run with median total runtime among three
trials and is normalized by the total SI runtime for the same problem.
``Orth / Whitening'' covers basis orthonormalization for SI and whitening for sSI.
``Basis update'' includes basis transformations and, for sSI,
updates of the maintained sketches.
For sSI, ``Orth / Whitening'' and ``RR / sRR'' also include the
final standard QR factorization and RR procedure, respectively.}
\label{fig:SI_ratio_cluster}
\end{figure}

As reported in the ``save'' column of Table~\ref{tab:SI},
sSI typically reduces runtime by \(10\%\)--\(40\%\),
with reductions of \(39\%\) for both \texttt{mario001}
and \texttt{c-65} when \(\nu=2\).
For the smallest matrix, \texttt{bcsstm21}, however,
the sketching overhead outweighs the computational savings,
making sSI slower than SI.

It is also worth noting that for larger problems such as \texttt{Andrews} and \texttt{Ga10As10H30},
sSI yields more modest runtime reductions.
As shown in Figure~\ref{fig:SI_ratio_cluster},
the linear solves \(\A^{-1}\X\) dominate the runtime
for these two problems.
Consequently, savings in the remaining components have
only a limited effect on the total runtime.

Another observation from  Table~\ref{tab:SI} is that increasing \(\nu\) from \(2\) to \(8\)
increases the runtime on most test problems.
This is because sSI with \(\nu=2\) is good enough to converge in nearly the same number of iterations as SI.
And larger sketch dimensions offer little reduction in iterations but increase the cost of sketch application,
basis whitening, and other operations in the sketched space.%
\footnote{For some problems, runs with larger sketches~(\(\nu=4\) or \(8\)) 
can be faster than those with \(\nu=2\),
despite requiring more outer iterations; 
see, for example, \texttt{Andrews} in Figure~\ref{fig:SI_ratio_cluster}.
This behavior is explained by soft locking:
differences in the locking histories can reduce the number
of active columns, potentially lowering the cost of
subsequent linear solves and basis whitening.
However, these savings are specific to the observed runs
and do not represent a general trend.}
These results support the choice \(\skdim=2\nex\) on the
tested problems, as it yields iteration counts close to those
of SI while limiting the cost of operations in the sketched space.

\subsection{Experiments on TraceMIN and sTraceMIN} \label{subsec:experiment_TrMIN}

Table~\ref{tab:TrMIN} reports the performance of TraceMIN
and sTraceMIN on the twelve test problems.
The results are similar to those for SI and sSI
(except for \texttt{c-65}, on which all four configurations
fail to converge within \(1000\) iterations).
Although the final original-space relative residuals \(\resnorm\)
of sTraceMIN can be slightly larger than \(\tol\),
they do not exceed \(1.4\times10^{-10}\) on the converged problems.

\begin{table}[tbhp]
\centering
\caption{Runtime in seconds, outer iteration count and final relative
residual in the original space for TraceMIN and sTraceMIN.
Problem numbers follow Table~\ref{tab:test_matrices}.
We use \(\nev=100\), \(\nex=150\) and \(\skdim=\nu\nex\).
Each configuration is represented by the trial with median total runtime.
The column ``save'' gives the percentage of time saved relative to TraceMIN.
The italicized results for \texttt{c-65} indicate that all four
configurations reach the limit of \(1000\) iterations without convergence;
the reported savings compare runtimes at this fixed iteration limit.}
\resizebox{\textwidth}{!}{%
\begin{tabular}{lrrrrrrrrrrrrrrr}
\toprule
No. & \multicolumn{3}{c}{TrMIN} & \multicolumn{4}{c}{sTrMIN ($\nu$ = 2)} & \multicolumn{4}{c}{sTrMIN ($\nu$ = 4)} & \multicolumn{4}{c}{sTrMIN ($\nu$ = 8)} \\
\cmidrule(lr){2-4}\cmidrule(lr){5-8}\cmidrule(lr){9-12}\cmidrule(lr){13-16}
 & time & iter & relres & time & iter & relres & save & time & iter & relres & save & time & iter & relres & save \\
\midrule
1 & 1.8e+00 & 30 & 8.0e-11 & 2.5e+00 & 31 & 6.6e-11 & -35\% & 2.2e+00 & 31 & 5.5e-11 & -23\% & 2.4e+00 & 30 & 9.5e-11 & -35\% \\
2 & 1.5e+01 & 193 & 9.2e-11 & 1.9e+01 & 191 & 1.4e-10 & -31\% & 1.6e+01 & 193 & 1.0e-10 & -6\% & 1.7e+01 & 193 & 1.0e-10 & -14\% \\
3 & 1.1e+01 & 117 & 8.7e-11 & 1.1e+01 & 117 & 1.3e-10 & 6\% & 8.9e+00 & 116 & 1.1e-10 & 21\% & 9.6e+00 & 116 & 9.6e-11 & 15\% \\
4 & 1.9e+01 & 188 & 9.4e-11 & 1.8e+01 & 187 & 1.2e-10 & 9\% & 1.8e+01 & 187 & 1.1e-10 & 10\% & 1.9e+01 & 188 & 9.8e-11 & 4\% \\
5 & 1.9e+01 & 224 & 9.8e-11 & 1.9e+01 & 222 & 1.3e-10 & 4\% & 1.9e+01 & 224 & 1.0e-10 & 4\% & 1.9e+01 & 224 & 1.0e-10 & 3\% \\
6 & 3.8e+01 & 237 & 9.5e-11 & 3.2e+01 & 234 & 1.2e-10 & 15\% & 3.2e+01 & 233 & 1.1e-10 & 15\% & 3.3e+01 & 239 & 9.6e-11 & 11\% \\
7 & 2.8e+01 & 113 & 9.5e-11 & 2.4e+01 & 111 & 1.2e-10 & 15\% & 2.4e+01 & 111 & 1.0e-10 & 15\% & 2.5e+01 & 113 & 9.2e-11 & 11\% \\
8 & 1.6e+02 & 317 & 9.8e-11 & 1.3e+02 & 307 & 1.3e-10 & 21\% & 1.3e+02 & 310 & 1.1e-10 & 19\% & 1.3e+02 & 315 & 1.0e-10 & 18\% \\
9 & \textit{8.3e+02} & \textit{1000} & \textit{6.7e-07} & \textit{6.6e+02} & \textit{1000} & \textit{8.0e-07} & \textit{21\%} & \textit{6.7e+02} & \textit{1000} & \textit{6.8e-07} & \textit{19\%} & \textit{6.8e+02} & \textit{1000} & \textit{6.8e-07} & \textit{18\%} \\
10 & 8.3e+01 & 87 & 9.2e-11 & 6.7e+01 & 86 & 1.2e-10 & 19\% & 6.7e+01 & 87 & 9.9e-11 & 20\% & 6.8e+01 & 88 & 9.5e-11 & 19\% \\
11 & 2.1e+02 & 158 & 9.8e-11 & 1.8e+02 & 171 & 1.3e-10 & 14\% & 1.8e+02 & 167 & 1.0e-10 & 12\% & 1.8e+02 & 166 & 9.7e-11 & 13\% \\
12 & 3.4e+02 & 190 & 9.3e-11 & 2.8e+02 & 194 & 1.3e-10 & 20\% & 2.7e+02 & 194 & 1.0e-10 & 20\% & 2.7e+02 & 192 & 9.7e-11 & 21\% \\
\bottomrule
\end{tabular}
}
\label{tab:TrMIN}
\end{table}

The outer iteration counts of sTraceMIN remain close to those
of TraceMIN on the converged problems, with a maximum relative
increase of approximately \(8.2\%\).
For all three sketch dimensions, sTraceMIN reduces the runtime
on nine of the eleven converged problems,
with savings typically around \(10\%\)--\(20\%\).

Although sTraceMIN yields smaller relative runtime reductions
than sSI on most test problems, it offers larger reductions
on some of the larger problems.
For \texttt{Andrews} and \texttt{Ga10As10H30},
sTraceMIN reduces runtime by \(19\%\)--\(21\%\) relative to
TraceMIN, compared with \(8\%\)--\(14\%\) for sSI relative to SI.
The timing breakdowns in
Figures~\ref{fig:SI_ratio_cluster}
and~\ref{fig:TrMIN_ratio_cluster} help explain this difference.
For these two problems, the correction stage using \(5\)
GMRES iterations accounts for a substantially smaller
fraction of the TraceMIN runtime than the direct solves
for \(\A^{-1}\X\) do in SI.
Consequently, savings in orthonormalization, RR, and
residual evaluation have a greater impact on the relative
reduction in total runtime.
These results suggest complementary roles for the two methods,
with their relative efficiency depending on both the cost
of the linear solves and the number of outer iterations.

\begin{figure}[tbhp]
\centering
\includegraphics[width=0.98\linewidth]{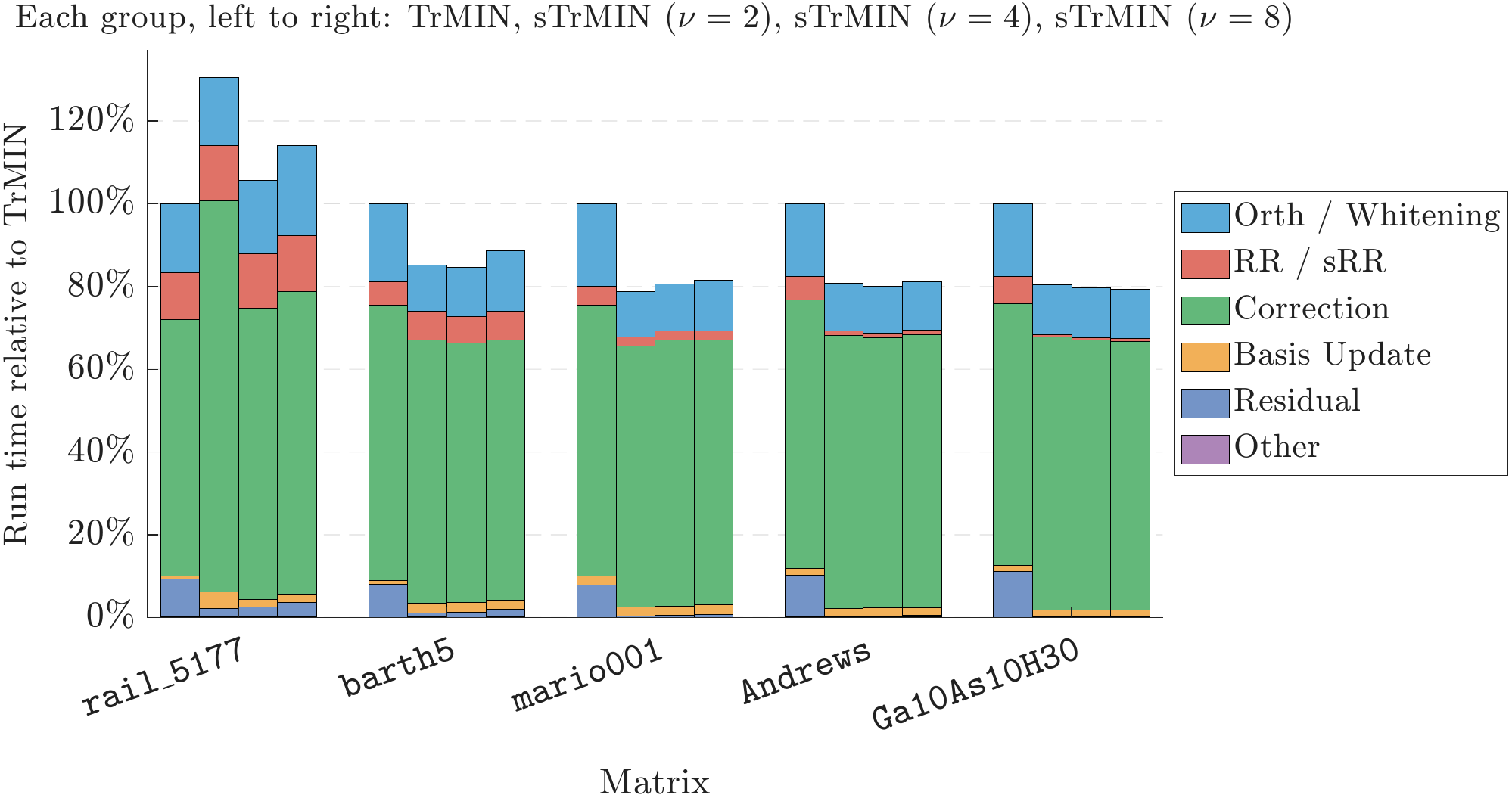}
\caption{Runtime breakdown for \texttt{rail\_5177},
\texttt{barth5}, \texttt{mario001},
\texttt{Andrews} and \texttt{Ga10As10H30}
in Table~\ref{tab:test_matrices}.
Within each group, the bars show TraceMIN and sTraceMIN with
\(\nu=2,4,8\), from left to right.
Both methods compute \(\nev=100\) smallest eigenpairs using
\(\nex=150\) search vectors.
For sTraceMIN, \(\skdim=\nu\nex\) is the sketch dimension.
Each bar represents the run with median total runtime among three
trials and is normalized by the total TraceMIN runtime for the same problem.
``Orth / Whitening'' covers basis orthonormalization for TraceMIN
and whitening for sTraceMIN.
``RR / sRR'' covers RR or sRR procedures.
``Correction'' covers forming the projected correction equations,
solving them by \(5\) GMRES iterations, and reprojecting and applying
the corrections.
``Basis Update'' includes basis transformations and, for sTraceMIN,
updates of the maintained sketches.
For sTraceMIN, ``Orth / Whitening'' and ``RR / sRR'' also include
the final standard QR factorization and RR procedure, respectively.}
\label{fig:TrMIN_ratio_cluster}
\end{figure}

\begin{remark}[Potential savings for sTraceMIN]
It is worth noting that sTraceMIN has the potential for
larger speedups than those observed above.
When solving the correction equations, GMRES repeatedly
applies \(\TrMINP\) in TraceMIN.
These projections typically account for \(20\%\)-\(50\%\) of total time of GMRES.
Replacing \(\TrMINP\) with \(\TrMINPsk\) in sTraceMIN therefore 
is expected to yield large runtime reductions
in sTraceMIN.
For \(\W\in\bbR^{\n\times\nrhs}\), the leading costs of
\(\TrMINP\W\) and \(\TrMINPsk\W\) are \(4\n\nex\nrhs\) and
\(2\n\nex\nrhs+16\n\nrhs+2\skdim\nex\nrhs\) flops, respectively.
Thus, when \(\skdim\ll\n\) and \(\nex \gg 8\), sketching
reduces the leading arithmetic cost of projection by nearly half.

Nevertheless, as shown in Figure~\ref{fig:TrMIN_projection_components},
the reduction in arithmetic operations does not
translate into comparable runtime savings.
The main reason is the relatively high cost of computing
\(\skmat\W\), which requires only \(16\n\nrhs\) flops
but can take as much time as, or even more time than \(\X\trans\W\),
which requires \(2\n\nex\nrhs\) flops.
This is partly due to the low arithmetic intensity
of sparse matrix multiplication, whose performance is often
limited by memory access rather than arithmetic throughput.
Efficient implementations of sparse-sign multiplication
have been developed~\cite{Yaniv_Ghysels_Malik_Boateng_Li_2025}.
Given the arithmetic cost ratio \(8/\nex\), such techniques
may offer a way to make computing \(\skmat\W\)
substantially cheaper than computing \(\X\trans\W\).
This could allow more of the arithmetic savings of the
sketched projections to translate into runtime reductions
in the GMRES correction solves.
\end{remark}

\begin{figure}[tb]
\centering
\includegraphics[width=0.98\linewidth]{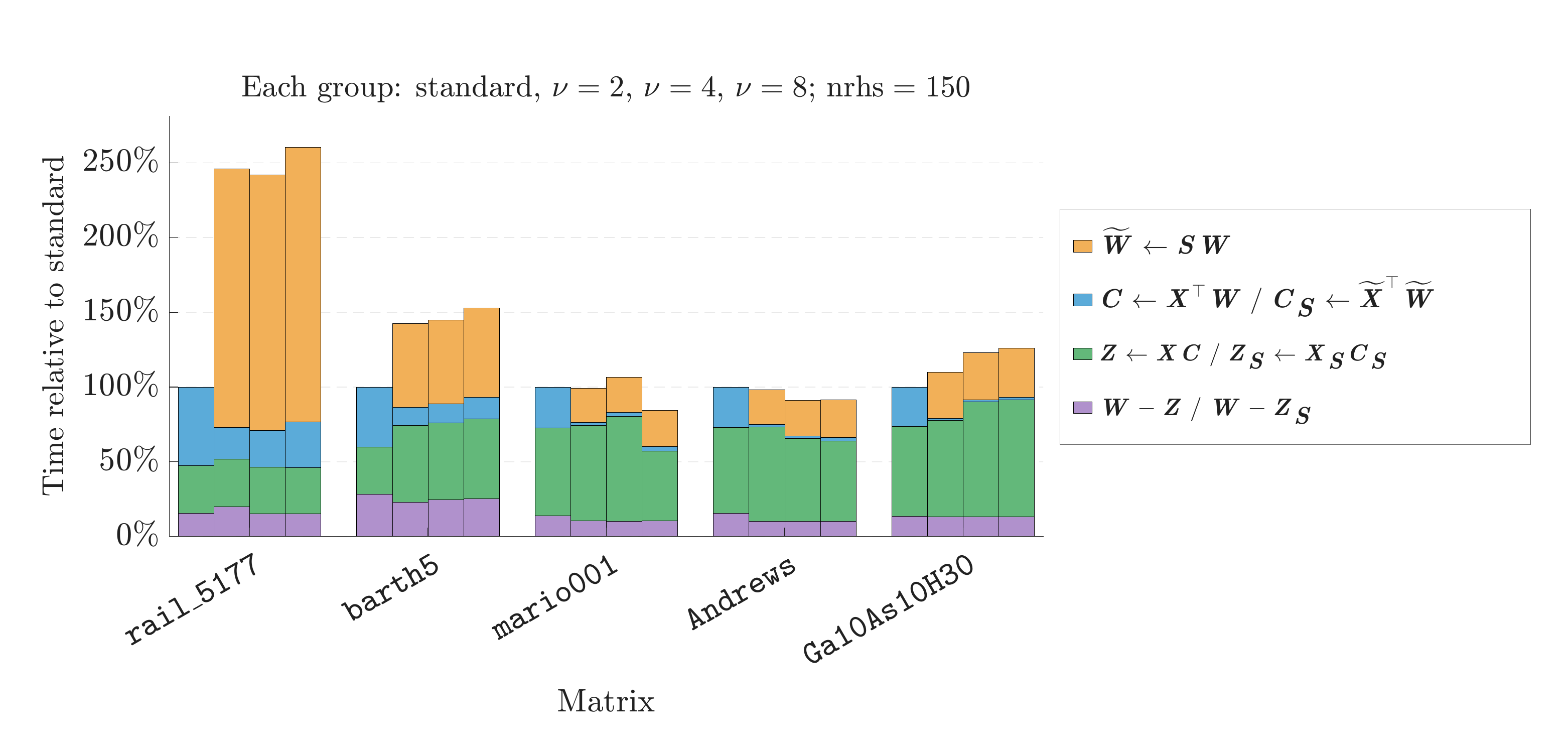}
\caption{Timing components of the standard projection
\(\TrMINP\W=\W-\X(\X\trans\W)\)
and the sketched projection
\(\TrMINPsk\W=\W-\XS(\skX\trans\skmat\W)\),
measured independently of the eigensolvers,
excluding the construction and preparation of the bases
and sketching matrices.
Here, \(\X,\XS\in\bbR^{\n\times\nex}\) span the same subspace and satisfy
\(\X\trans\X=\I_{\nex}\) and \(\skX\trans\skX=\I_{\nex}\),
where \(\skX=\skmat\XS\).
The input is a Gaussian block \(\W\in\bbR^{\n\times\nrhs}\).
We use \(\nex=150\) basis vectors and \(\nrhs=150\) input vectors.
The row dimensions \(\n\) correspond, from left to right, to the
sizes of \texttt{rail\_5177}, \texttt{barth5}, \texttt{mario001},
\texttt{Andrews} and \texttt{Ga10As10H30}
in Table~\ref{tab:test_matrices}.
Each group shows the standard projection followed by the
sketched projections with \(\nu=2,4,8\), where \(\skdim=\nu\nex\).
The sketched coefficient product uses the already computed
\(\tW=\skmat\W\).
Component times are averaged over ten repetitions.
Each sketched bar uses the run with median total projection time
among three trials.
The standard reference uses the run with median total projection
time among three independent standard-projection trials.
All bars are normalized by this reference run's total projection time.}
\label{fig:TrMIN_projection_components}
\end{figure}

\section{Conclusions} \label{sec:conclusions}

In this work, we develop two sketched eigensolvers for real symmetric
eigenvalue problems, sSI and sTraceMIN, by incorporating basis whitening
and sRR into SI and TraceMIN, respectively.
To avoid the complex arithmetic that may arise from the nonsymmetric
projected matrix in sRR, we develop a technique based on the real Schur
decomposition for characteristic subspace selection and residual estimation.
Combined with basis whitening, this framework preserves real
arithmetic and replaces the expensive Euclidean orthogonalization of the search
basis with a cheaper orthogonalization in the sketch space.

On the theoretical side, we provide, to our knowledge, the first
systematic investigation of suitable subspace embedding
assumptions for sketched eigensolvers that use a fixed sketching matrix.
Motivated by this analysis, we propose a local embedding assumption
requiring \(\skmat\) to be a subspace embedding only for the subspace \(\Span\{\V_1\}\) 
associated with the \(\nex\) smallest eigenvalues, 
rather than for all intermediate search subspaces.
Under this embedding assumption,
we establish local convergence of both sSI and sTraceMIN.

Numerical experiments demonstrate that sketching can substantially
reduce the costs of basis orthogonalization, residual evaluation, 
and, potentially, RR process.
Both sSI and sTraceMIN reduce the overall runtime on most converged
test problems, with savings typically
\(10\%\)--\(30\%\) for sSI and \(10\%\)--\(20\%\) for sTraceMIN.
For sSI, the savings exceed \(30\%\) in some cases.
The experiments with different sketch dimensions also show
that relatively small sketches can be effective on the sketched eigensolvers,
which aligns with our theoretical analysis.

Our future work will focus on extending the basis whitening, sRR, and
real Schur framework developed in this paper to a broader class of
eigensolvers, including JD and LOBPCG algorithms.
We also plan to develop a high-performance numerical software library
targeting modern parallel computing architectures.
We expect that more efficient implementations,
e.g., faster routines for applying sparse sign sketching matrices,
could further improve the performance of the proposed algorithms.

\section*{Acknowledgments}
JL thanks the China Scholarship Council~(CSC)
for supporting his visit to The University of Manchester.
YL was supported by the Engineering and Physical
Sciences Research Council [grant number UKRI2774].
We thank Massimiliano Fasi for valuable discussions.
We also acknowledge the use of GPT-6 to help compute the bounds in Theorems~\ref{thm:local_equiv_TraceMIN_sTraceMIN} and~\ref{thm:local_gap_RR_sRR}
and to improve the clarity and readability of the manuscript.
All mathematical results and proofs were independently verified by the authors.

\appendix
\section{Standard RR, SI and TraceMIN Algorithms} \label{app:RR_SI_TraceMIN}

We present the standard RR, SI and TraceMIN algorithms here for reference;
see Algorithms~\ref{alg:RR}, \ref{alg:SI}, and~\ref{alg:TraceMIN}.

\begin{algorithm}[ht!]
  \caption{Rayleigh--Ritz procedure} \label{alg:RR}
  \begin{algorithmic}[1]
    \FUNCTION{\([\Z, \Th] = \RRfunc(\A, \X)\)}
    \REQUIRE{Orthonormal \(\X\)}
    \ENSURE{Orthogonal matrix \(\Z\) and diagonal matrix \(\Th\)}
    \STATE{\(\Ap = \X\trans \A \X\)}
    \STATE{Compute an eigendecomposition \(\Ap = \Z \Th \Z\trans\)}
    \STATE{Sort the diagonal entries of \(\Th\) in nondecreasing order and reorder the columns of \(\Z\) accordingly}
  \end{algorithmic}
\end{algorithm}

\begin{algorithm}[ht!]
  \caption{Subspace iteration} \label{alg:SI}
  \begin{algorithmic}[1]
    \FUNCTION{\([\V, \Lam] = \SIfunc(\A, \X)\)}
    \REQUIRE{SPD matrix \(\A\), initial guess \(\X\)}
    \ENSURE{Approximate eigenpairs \((\V, \Lam)\)}
    \STATE{\([\X, \sim] = \QRfunc(\X)\)}
    \STATE{\([\Z, \Th] = \RRfunc(\A, \X)\)}
    \STATE{\(\X = \X \Z\)}
    \STATE{\(\Rused = \A \X(:, 1 : \nev) - \X(:, 1 : \nev) \Th(1 : \nev, 1 : \nev)\)}
    \STATE{Check convergence using \(\Rused\)}
    \WHILE{not converged}
      \STATE{\(\X \gets \A^{-1} \X\)}
      \STATE{\([\X, \sim] = \QRfunc(\X)\)}
      \STATE{\([\Z, \Th] = \RRfunc(\A, \X)\)}
      \STATE{\(\X = \X \Z\)}
      \STATE{\(\Rused = \A \X(:, 1 : \nev) - \X(:, 1 : \nev) \Th(1 : \nev, 1 : \nev)\)}
      \STATE{Check convergence using \(\Rused\)}
    \ENDWHILE
    \STATE{\(\V = \X(:, 1 : \nev)\), \(\Lam = \Th(1 : \nev, 1 : \nev)\)}
  \end{algorithmic}
\end{algorithm}

\begin{algorithm}[ht!]
  \caption{TraceMIN} \label{alg:TraceMIN}
  \begin{algorithmic}[1]
    \FUNCTION{\([\V, \Lam] = \TrMINfunc(\A, \X)\)}
    \REQUIRE{SPD matrix \(\A\), initial guess \(\X\)}
    \ENSURE{Approximate eigenpairs \((\V, \Lam)\)}
    \STATE{\([\X, \sim] = \QRfunc(\X)\)}
    \STATE{\([\Z, \Th] = \RRfunc(\A, \X)\)}
    \STATE{\(\X = \X \Z\)}
    \STATE{\(\R = \A \X - \X \Th\) and \(\Rused = \R(:, 1 : \nev)\)}
    \STATE{Check convergence using \(\Rused\)}
    \WHILE{not converged}
      \STATE{Construct the projector \(\TrMINP = \I_{\n} - \X \X\trans\)}
      \STATE{Approximately solve \(\TrMINP \A \TrMINP \TrMINDel = \TrMINP \R\) and update \(\TrMINDel = \TrMINP \TrMINDel\)}
      \STATE{\(\X = \X - \TrMINDel\)}
      \STATE{\([\X, \sim] = \QRfunc(\X)\)}
      \STATE{\([\Z, \Th] = \RRfunc(\A, \X)\)}
      \STATE{\(\X = \X \Z\)}
      \STATE{\(\R = \A \X - \X \Th\) and \(\Rused = \R(:, 1 : \nev)\)}
      \STATE{Check convergence using \(\Rused\)}
    \ENDWHILE
    \STATE{\(\V = \X(:, 1 : \nev)\), \(\Lam = \Th(1 : \nev, 1 : \nev)\)}
  \end{algorithmic}
\end{algorithm}

\end{document}